\documentclass[a4paper,12pt]{amsart}
\usepackage[left=3.4 cm,top=2.4cm,bottom=3.4cm,right=3.4cm]{geometry}
\usepackage{amsfonts,amssymb,amsthm,amsmath}
\usepackage{bbm}

\usepackage{xcolor}

\definecolor{green(munsell)}{rgb}{0.0, 0.66, 0.47}
\definecolor{BlueGreenn}{rgb}{0.3,0.5,0.8}
\definecolor{DB}{rgb}{0.3,0.3,0.3}
\definecolor{DOr}{rgb}{0.7,0.3,0.3}
\definecolor{DGr}{rgb}{0.3,0.7,0.3}
\definecolor{DBl}{rgb}{0.1,0.3,0.5}
\definecolor{arylideyellow}{rgb}{0.91, 0.84, 0.42}
\definecolor{burntorange}{rgb}{0.8, 0.33, 0.0}
\definecolor{chromeyellow}{rgb}{1.0, 0.65, 0.0}
\usepackage[hypertexnames=true, citecolor = burntorange, colorlinks = true, linkcolor = BlueGreenn, pdffitwindow = false, urlbordercolor = white,pdfborder={0 0 0}]{hyperref}%

\usepackage[T1]{fontenc}
\usepackage{microtype}

\usepackage{color}
\usepackage[utf8]{inputenc}

\usepackage{framed}
\usepackage{centernot}

\usepackage{tikz-cd} 
\usepackage{tkz-euclide}

\numberwithin{equation}{section}

\newtheorem{theorem}{Theorem}[section]
\newtheorem{proposition}[theorem]{Proposition}
\newtheorem{lemma}[theorem]{Lemma}

\newtheorem{corollary}[theorem]{Corollary}
\theoremstyle{definition}
\newtheorem{definition}[theorem]{Definition}

\newtheorem{example}[theorem]{Example}

\theoremstyle{remark}
\newtheorem{remark}[theorem]{Remark}

\newcommand{\ot}{\otimes}
\newcommand{\tp}[1]{^{\otimes #1}}    

\DeclareMathOperator{\End}{End}

\DeclareMathOperator{\id}{id}
\DeclareMathOperator{\supp}{supp}
\DeclareMathOperator{\Tr}{Tr}

\DeclareMathOperator{\dom}{dom}

\newcommand{\Cl}{\mathbb{C}}
\newcommand{\Rl}{\mathbb{R}}
\newcommand{\Nl}{\mathbb{N}}

\newcommand{\B}{\mathcal{B}}

\newcommand{\M}{\mathcal{M}}

\newcommand{\Hil}{\mathcal{H}}

\newcommand{\Om}{\Omega}
\newcommand{\om}{\omega}
\newcommand{\la}{\lambda}
\newcommand{\eps}{\varepsilon}
\newcommand{\te}{\theta}

\usepackage[shortlabels]{enumitem}

\usepackage{mathtools}
\mathtoolsset{showonlyrefs=true} 

\usepackage{ytableau}
\ytableausetup{boxsize=0.5em,centertableaux}



\usepackage{comment} 

\newif\ifshow 
\showfalse 

\ifshow
  \includecomment{wrap}
\else
  \excludecomment{wrap} 
\fi

\newcommand{\bR}{{\mathbb R}}
\newcommand{\bC}{{\mathbb C}}
\newcommand{\Hardy}{{\mathbb H}}

\def\End{\operatorname{End}}
\def\Hom{\operatorname{Hom}}
\def\Rep{\operatorname{Rep}}
\def\Hilbf{\operatorname{Hilb_f}}

\def\ran{\operatorname{ran}}
\def\supp{\operatorname{supp}}
\def\ker{\operatorname{ker}}

\def\dim{\operatorname{dim}}

\def\id{\operatorname{id}}

\def\Re{\operatorname{Re}}
\def\Im{\operatorname{Im}}

\def\id{\operatorname{id}}
\def\supp{\operatorname{supp}}
\def\ev{\operatorname{ev}}
\def\coev{\operatorname{coev}}

\newcommand{\Strip}{{\mathbb S}}

\newcommand{\La}{\Lambda}

\newcommand{\CA}[0]{\mathcal{A}} \newcommand{\CB}[0]{\mathcal{B}}
\newcommand{\CD}[0]{\mathcal{D}}
\newcommand{\CE}[0]{\mathcal{E}} \newcommand{\CF}[0]{\mathcal{F}}
 
\newcommand{\CI}[0]{\mathcal{I}} \newcommand{\CJ}[0]{\mathcal{J}}
 
\newcommand{\CM}[0]{\mathcal{M}} \newcommand{\CN}[0]{\mathcal{N}}

\newcommand{\CS}[0]{\mathcal{S}} 
\newcommand{\CU}[0]{\mathcal{U}} 
 \newcommand{\CX}[0]{\mathcal{X}}

\newcommand{\ip}[2]{\left\langle#1,#2\right\rangle}

\usepackage{mathrsfs}

\DeclareMathOperator{\Cross}{Cr}
\DeclareMathOperator{\CatCross}{CatCr}
\DeclareMathOperator{\St}{St}
\newcommand{\Deltah}{\widehat{\Delta}}
\newcommand{\Deltat}{\widetilde{\Delta}}
\newcommand{\Deltatin}{\widetilde{\Delta}^{\rm in}}
\newcommand{\Deltatout}{\widetilde{\Delta}^{\rm out}}\newcommand{\Deltatinout}{\widetilde{\Delta}^{\rm in/out}}
\newcommand{\Deltatoutin}{\widetilde{\Delta}^{\rm out/in}}
\newcommand{\Jh}{\widehat{J}\,}
\newcommand{\JF}{\widetilde{J}_F}
\newcommand{\Sh}{\widehat{S}\,}
\newcommand{\Hh}{\widehat{H}}
\newcommand{\CrossSet}{\mathcal{C}}
\newcommand{\CrossSetb}{\mathcal{C}_{\text{b}}}

\title{Crossing symmetry and the crossing map}
\author{Ricardo Correa da Silva, Luca Giorgetti, Gandalf Lechner}
\address[1]{Department Mathematik, FAU Erlangen-N\"urnberg, Cauerstr. 11, DE-91058 Erlangen, Germany}
\email{ricardo.correa.silva@fau.de}
\email{gandalf.lechner@fau.de}
\address[2]{Dipartimento di Matematica, Universit\`a di Roma Tor Vergata, Via della Ricerca Scientifica 1, I-00133 Roma, Italy}
\email{giorgett@mat.uniroma2.it}

\thanks{
L.G. is partially supported by \lq\lq MIUR Excellence Department Project MatMod@TOV\rq\rq\ awarded to the Department of Mathematics, University of Rome Tor Vergata, CUP E83C23000330006, by the University of Rome Tor Vergata funding OAQM, CUP E83C22001800005, by progetto GNAMPA 2023 \lq\lq Metodi di Algebre di Operatori in Probabilit\`a non Commutativa\rq\rq\ CUP E53C22001930001, and by progetto GNAMPA 2024 \lq\lq Probabilit\`a Quantistica e Applicazioni\rq\rq\ CUP E53C23001670001.
G.L. is supported by the Deutsche Forschungsgemeinschaft DFG through the Heisenberg project ``Quantum Fields and Operator Algebras'' (LE 2222/3-1).}

\date{}

\begin{document}

\begin{abstract}
	We introduce and study the crossing map, a closed linear map acting on operators on the tensor square of a given Hilbert space that is inspired by the crossing property of quantum field theory. This map turns out to be closely connected to Tomita--Takesaki modular theory. In particular, crossing symmetric operators, namely those operators that are mapped to their adjoints by the crossing map, define endomorphisms of standard subspaces. Conversely, such endomorphisms can be integrated to crossing symmetric operators.
	We also investigate the relation between crossing symmetry and natural compatibility conditions with respect to unitary representations of certain symmetry groups, and furthermore introduce a generalized crossing map defined by a real object in an abstract $C^*$-tensor category, not necessarily consisting of Hilbert spaces and linear maps. This latter crossing map turns out to be closely related to the (unshaded, finite-index) subfactor theoretical Fourier transform. Lastly, we provide families of solutions of the crossing symmetry equation, solving in addition the categorical Yang--Baxter equation, associated with an arbitrary Q-system.
\end{abstract}

\maketitle

\section{Introduction}

In quantum field theory, the crossing property or crossing symmetry is a symmetry relating different scattering amplitudes by analytic continuation in the Mandelstam variables of the corresponding process. While crossing symmetry seems to be a fundamental property of QFT \cite{Schroer:2010,Mizera:2021}, it is difficult to derive it from first principles (see, however, \cite{BrosEpsteinGlaser:1965}) and it is therefore often taken as an assumption for concrete models (see, for example, \cite{Martin:1969_2,PolandSimmons-Duffin:2016}).

Crossing symmetry relates the scattering amplitude of a $2\to 2$ process with the amplitude of a process in which two of the particles are exchanged with their antiparticles (TCP conjugates). Diagrammatically, this connects the two diagrams
\begin{figure}[h!]
	\includegraphics[width=50mm]{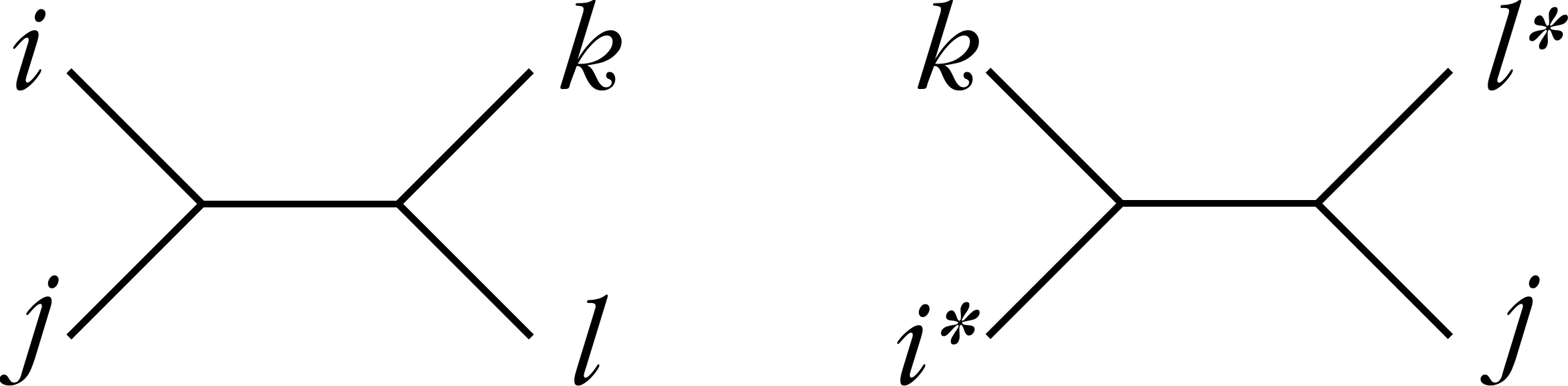}
	\caption{}\label{fig:diagrams}
\end{figure}

\noindent in which asterisks denote antiparticles. On a diagrammatic level, crossing symmetry therefore amounts to an operation involving an object with four indices that are rotated by $90^\circ$ and possibly conjugated. Such operations have also been considered in purely mathematical settings independent of QFT -- see, for example, \cite{Lyubashenko:1987,BozejkoSpeicher:1994,JorgensenSchmittWerner:1995}.

While this diagrammatic point of view might appear to be rather superficial regarding the crossing symmetry from QFT, we will show in this article that it allows for a rich mathematical framework which includes the QFT crossing symmetry as a special case.

To motivate our setting of crossing symmetry, we consider a complex Hilbert space $\Hil$ and postpone a proper discussion of all functional analytic aspects pertaining to domains and boundedness to the body of the paper. Reading the diagrams in Figure~\ref{fig:diagrams} according to the usual graphical notation for tensor products, we aim to associate with linear operators $T:\Hil\ot\Hil\to\Hil\ot\Hil$ their ``crossed versions'' $\widehat{T}:\Hil\ot\Hil\to\Hil\ot\Hil$ which ought to be defined in terms of their matrix elements $\widehat T^{\varphi_1\varphi_2}_{\psi_2\psi_1} := \langle\varphi_1\ot\varphi_2,\widehat{T}(\psi_2\ot\psi_1)\rangle$ by a rotation of the four indices by $90^\circ$, $\widehat{T}^{\varphi_1\varphi_2}_{\psi_2\psi_1} = T^{\varphi_2\psi_1}_{\varphi_1\psi_2}$, in a suitable basis, or,
\begin{align}\label{eq:cross-idea1}
	\langle\varphi_1\ot\varphi_2,\widehat{T}(\psi_2\ot\psi_1)\rangle
	&=
	\langle\varphi_2\ot S^*\psi_1,T(S\varphi_1\ot\psi_2)\rangle, \qquad\varphi_k,\psi_k\in\Hil.
\end{align}
In \eqref{eq:cross-idea1}, we have included an {\em antilinear} map $S$ and its adjoint $S^*$ which is clearly necessary for the equation to make sense. Crossing symmetric operators~$T$ should then be those satisfying $\widehat{T}=T^*$.

To specify the antilinear operator~$S$ defining the map $T\mapsto\widehat T$, we take as a guiding principle that the tensor flip $F:\Hil\ot\Hil\to\Hil\ot\Hil$, $F(\varphi_1\ot\varphi_2)=\varphi_2\ot\varphi_1$, should be crossing symmetric, $\widehat{F}=F^*$. This is on the one hand in line with what is suggested by the diagrams, and on the other hand justified by the distinguished role played by the flip in the quantum field theoretic context, where it identifies the free theory \cite{CorreaDaSilvaLechner:2023}. It then follows from a simple calculation based on \eqref{eq:cross-idea1} that $S$ needs to be an {\em involution}.

We thus see that the idea of crossing symmetry on Hilbert spaces automatically leads us to consider antilinear involutions. As will be explained in the next section, it is well possible and interesting to consider general (not necessarily bounded) densely defined closed antilinear involutions, which immediately indicates that crossing symmetry should be considered in the context of Tomita--Takesaki modular theory. In view of the close connection of some modular operators to TCP symmetry \cite{BisognanoWichmann:1976,Borchers:1992}, and taking into account that the crossing property can sometimes be proven by a process of analytic continuation \cite{BrosEpsteinGlaser:1965} resembling the KMS/modular condition, it is reasonable to expect that also the abstract crossing we study here should be connected to modular theory, as is indeed the case. We refer to \cite{Niedermaier:1998,Lechner:2003,Schroer:2010,BischoffTanimoto:2013,HollandsLechner:2018} for previous connections between crossing symmetry and modular theory.

One of the main motivations for the present article comes from our recent work \cite{CorreaDaSilvaLechner:2023}. In that paper, we construct a so-called twisted Araki-Woods von Neumann algebra, represented on a twisted Fock space, based on two objects: A selfadjoint twist operator $T\in\CB(\Hil\ot\Hil)$ and a standard subspace (cf.\ Remark~\ref{remark:modulartheory}) $H\subset\Hil$. In case these two data are compatible in the sense that the modular unitaries $\Delta_H^{it}\ot\Delta_H^{it}$ commute with $T$, it was shown that the Fock vacuum is cyclic and separating for the twisted algebra if and only if $T$ satisfies a crossing symmetry condition w.r.t.\ the Tomita operator $S_H$ of $H$, and solves the Yang--Baxter equation. As these algebras underlie the construction of many quantum field theoretic models, one is interested in examples of crossing symmetric operators.

The main purpose of this article is a systematic investigation of a crossing map that is well suited for Hilbert space operators and modular theory. This investigation thus connects two areas in which Huzihiro Araki made lasting contributions (as a sample of his work, see \cite{ArakiHaag:1967,Araki:1974}), as appropriate for this memorial volume. Also the main applications that we have in mind, namely a better understanding of the twisted Araki--Woods algebras generalizing the (untwisted) Araki--Woods factors \cite{ArakiWoods:1968} (and its already existing generalizations \cite{Shlyakhtenko:1997,KumarSkalskiWasilewski:2023,RahulKumarWirth:2024}), connect to Araki's work. In this article, we however concentrate on the crossing map as such and postpone applications to future work.

\medskip

In Section~\ref{sec:crossingS}, we give a detailed description of the crossing map that was sketched above. We start from a closed densely defined antilinear involution~$S$ on a Hilbert space~$\Hil$, and define a crossing map $T\mapsto\Cross_S(T)$ for operators $T\in\CB(\Hil\ot\Hil)$ (Def.~\ref{def:Scrossable}). For infinite-dimensional $\Hil$, the map $\Cross_S$ is however only densely defined (not all $T\in\B(\Hil\ot\Hil)$ can be ``crossed'') and some care is needed when studying its properties. In Section~\ref{sec:crossingS}, we therefore analyze closedness properties of $\Cross_S$ (Prop.~\ref{prop:generalcrossing}) and establish the link between crossing symmetry and KMS-type analytic conditions (Prop.~\ref{prop:KMS-likeCrossing}).

In Section~\ref{sec:crossingSbdd}, we consider the special case of a bounded involution $S$, including the subcase of antiunitary involutions. We show that for bounded $S$, the space of all crossing symmetric operators, i.e.\ all $T\in\CB(\Hil\ot\Hil)$ that lie in the domain of $\Cross_S$ and satisfy $\Cross_S(T)=T^*$, closely resembles a (Banach space version of a) standard subspace (Thm.~\ref{thm:Hhat}). When passing to Hilbert--Schmidt operators, it turns out that the crossing symmetric ones indeed form a standard subspace (Prop.~\ref{prop:HS-Crossing}). We thus obtain a map from standard subspaces $H\subset\Hil$ (given by involutions $S$ on $\Hil$, see Remark~\ref{remark:modulartheory} below) to standard subspaces $\Hh$ in the Hilbert space of Hilbert--Schmidt operators on $\Hil\ot\Hil$. As the crossing map is isometric in Hilbert--Schmidt norm, we find several formal analogies to Fourier transforms.

Building on this observation, we study in Section~\ref{sec:endos} a tight correspondence between crossing symmetric operators (w.r.t.\ a general involution $S$) and endomorphisms of the standard subspace corresponding to the chosen $S$. Our initial observation is that contraction of two tensor legs of an operator $T$ that is crossing symmetric w.r.t.\ $S$ yields an endomorphism of the standard subspace $H=\ker(S-1)$. Moreover, $T$ can be recovered from these endomorphisms (Thm.~\ref{thm:CrossingEnd}). This perspective allows for the construction of many crossing symmetric operators by integration of endomorphism-valued maps against spectral measures, generalizing previous work of Tanimoto \cite{Tanimoto:2011-1}. In particular, we here obtain examples of operators that are two-body S-matrices which are crossing symmetric in the QFT sense, thus linking back to the original idea of crossing symmetry (Remark~\ref{remark:ModvsQFT-Crossing}).

In applications, one is often interested in operators satisfying crossing symmetry and further conditions such as group symmetries. In Section~\ref{sec:symmetry}, we consider crossing symmetric operators that are invariant under a diagonal unitary representation of the form $g\mapsto U(g)\ot U(g)$ and study how symmetry of a crossing symmetric operator $T$ is reflected in its standard subspace endomorphisms. We determine the space of all crossing symmetric operators that are invariant under $G$-representations in some natural examples, namely the unitary group, the orthogonal group, and irreducible positive mass representations of the Poincaré group in two dimensions (Thm.~\ref{thm:CrossUPoincare}).

In Section~\ref{sec:catcross}, we study a generalized notion of crossing map and crossing symmetry in the context of abstract $C^*$-tensor categories (not necessarily consisting of Hilbert spaces and linear maps). Given a real self-conjugate object $X$ in such a category $C$, the categorical crossing map (Def.~\ref{def:CatCross}) is defined in terms of the associated solutions $\ev_X \in \Hom_C(X\otimes X, \id)$ and $\coev_X \in \Hom_C(\id, X\otimes X)$ of the conjugate equations, called respectively evaluation and coevaluation maps. Diagrammatically, the categorical crossing map has the form
\begin{figure}[h!]
	\includegraphics[width=40mm]{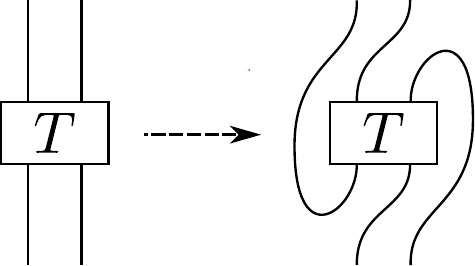}
	\caption{}\label{fig:crossing}
\end{figure}

This picture coincides with the one of Ocneanu's subfactor theoretical/planar algebraic Fourier transform \cite{Ocneanu:1991} (Remark \ref{rmk:CatCrossisF}), see also \cite{Bisch:1997,BischJones:2000,JaffeJiangLiuRenWu:2020} and references therein, and at the same time describes the Hilbert space crossing map studied before \cite[App.~A]{CorreaDaSilvaLechner:2023}.

We show in Theorem~\ref{thm:HilbfCrossing} that this is no coincidence -- For the category $\Hilbf$ of finite-dimensional Hilbert spaces, the categorical crossing map coincides with the Hilbert space crossing map considered in Sections~2--5. This theorem also characterizes standard subspaces of finite-dimensional Hilbert spaces in terms of $\ev_X$ and $\coev_X$.

There are, however, many scenarios that go beyond the Hilbert space setting. In particular, we show that any Q-system in a $C^*$-tensor category $C$ \cite{BischoffKawahigashiLongoRehren:2015} (a special $C^*$-Frobenius algebra object, originally studied in the context of subfactors \cite{Longo:1994}) defines a natural operator $T$ that is crossing symmetric w.r.t.\ the categorical crossing map  (Prop.~\ref{prop:TsolvesCross}) and moreover solves the tensor categorical Yang--Baxter equation (Prop.~\ref{prop:TsolvesYB}). If $C=\Hilbf$, then these $T$ defined by Q-systems provide (Cor.~\ref{cor:twistedAWalgs}) suitable twist operators to feed in the construction of $T$-twisted Araki--Woods algebras of \cite{CorreaDaSilvaLechner:2023}. We conclude the section by giving an explicit description of Q-systems in $\Hilbf$.

\begin{remark}\label{remark:modulartheory}
	As antilinear involutions will be important throughout the paper, we collect here some facts about spatial modular theory that we will use. See \cite{BratteliRobinson:1987,Longo:2008} for more details.
	
	We will consider densely defined closed antilinear involutions
	\begin{align}
		S:\dom(S)\subset\Hil\to\ran(S)\subset\Hil
	\end{align}
	on a complex separable Hilbert space $\Hil$. We denote the set of all such involutions $\CS(\Hil)$. Any $S\in\CS(\Hil)$ admits a polar decomposition $S=J\Delta^{1/2}$ with~$J$ antiunitary, $\Delta$ strictly positive, and the modular relation $J\Delta=\Delta^{-1}J$ holds. We then have $\dom(S)=\ran(S)$ and $S^*=J\Delta^{-1/2}$ is the polar decomposition of $S^*$. In particular, $S$ may be unbounded, bounded, or antiunitary (the latter case is equivalent to $\Delta=1$).
	Any antiunitary involution $J\in\CS(\Hil)$ can be described as complex conjugation in an orthonormal basis $\{e_n\}_{n\in\Nl}$ of $\Hil$, i.e.\ $J$ is given by $Je_n=e_n$ and antilinear extension.
	
	Given $S\in\CS(\Hil)$, the closed real linear space
	\begin{align}
		H_S := \ker(S-1)
	\end{align}
	has the two properties $H\cap iH=\{0\}$ and $\overline{H+iH}=\Hil$; it is called the \emph{standard subspace} associated to $S$. Any standard subspace (a closed real linear subspace $H\subset\Hil$ with the two mentioned properties) defines a unique $S_H\in\CS(\Hil)$, namely
	\begin{align}
		S_H:H+iH\to H+iH,\qquad h_1+ih_2\mapsto h_1-ih_2.
	\end{align}
	Thus, there is a bijection between $\CS(\Hil)$ and the set of standard subspaces of~$\Hil$.
	
	Unless mentioned otherwise, $\Hil$ denotes a complex separable Hilbert space, $S$ an element of $\CS(\Hil)$, and $S=J\Delta^{1/2}$ its polar decomposition.
\end{remark}

\section{The crossing map and crossing symmetry in Hilbert space}\label{sec:crossingS}

We begin by fixing $S\in\CS(\Hil)$ and describe its associated crossing map. For any bounded operator $T\in\B(\Hil\ot\Hil)$ on the tensor square Hilbert space, we define $Q_S(T)$ as the quadratic form with form domain $\dom(S)\odot\dom(S^*)$ (where $\odot$ denotes the algebraic tensor product) by
\begin{align}
	Q_S(T)(\varphi_1\ot\varphi_2,\psi_1\ot\psi_2)
	&:=
	\langle\varphi_2\ot S^*\psi_2,T(S\varphi_1\ot\psi_1)\rangle
\end{align}
and bi-sesquilinear extension.

\begin{definition}\label{def:Scrossable}
	Let $S\in\CS(\Hil)$ be a densely defined closed antilinear involution.
	\begin{enumerate}
		\item An operator $T\in\B(\Hil\ot\Hil)$ is called {\em $S$-crossable} if there exists a bounded operator $\Cross_S(T)\in\B(\Hil\ot\Hil)$ such that
		\begin{align*}
			Q_S(T)(\varphi_1\ot\varphi_2,\psi_1\ot\psi_2)
			&=
			\langle\varphi_1\ot\varphi_2,\Cross_S(T)(\psi_1\ot\psi_2)\rangle
		\end{align*}
		for all $\varphi_1,\psi_1\in \dom(S)$ and for all $\varphi_2,\psi_2\in \dom(S^*)$. This operator is then uniquely determined by $T$ and~$S$, and is called the $S$-{\em crossed version} of $T$.
		
		\item The set of all $S$-crossable operators is denoted $\CrossSet_S(\Hil\ot\Hil)\subset\B(\Hil\ot\Hil)$, or just $\CrossSet_S$ for short, and
		\begin{align}
			\Cross_S : \CrossSet_S(\Hil\ot\Hil) \to \B(\Hil\otimes \Hil)
		\end{align}
		is called the {\em $S$-crossing map}.
		
		\item An operator $T\in\CrossSet_S(\Hil\ot\Hil)$ is called {\em $S$-crossing symmetric} if
		\begin{align}
			\Cross_S(T)=T^*.
		\end{align}
	\end{enumerate}
\end{definition}

The following examples illustrate the concept of $S$-crossable operators.

\begin{example}\label{example:basics}
	\leavevmode
	\begin{enumerate}
		\item For any $S\in\CS(\Hil)$, the tensor flip $F(v\ot w):=w\ot v$ is $S$-crossing symmetric and a fixed point of the crossing map,
		\begin{align}
			\Cross_S(F)=F^*=F.
		\end{align}
		This follows by straightforward computation as indicated in the Introduction.
		
		\item \label{ex:ProductTimesFlip}Slightly more generally than in a), one finds by explicit calculation for any $S\in\CS(\Hil)$ and for any $A,B\in\B(\Hil)$ such that $S^*B^*S^*$ is bounded, we have $(A\ot B)F\in\CrossSet_S$ and
		\begin{align}
			\Cross_S((A\ot B)F)
			=
			F(A\ot S^*B^*S^*)
			=
			(S^*B^*S^* \ot A)F.
		\end{align}
		We therefore see that the crossing map is closely connected to the $S^*${\em -partial transpose} $A\ot B\mapsto A\ot S^*B^*S^*$. Note, however, that $A\ot B$ without the flip factor is in general not crossable if $\dim\Hil=\infty$, see, e.g.\ point d) below.
		
		\item The following example can be seen as a continuous analogue of the previous one. Let $\Hil:=L^2(\Rl)$, $(Sf)(x):=\overline{f(x)}$, and $w\in L^\infty(\Rl^2)$. Then the weighted flip
		\begin{align*}
			(F_w\psi)(x,y) := w(x,y)\psi(y,x)
		\end{align*}
		is $S$-crossable with $\Cross_S(F_w)=F_{\hat w}$, where $\hat w(x,y):=w(y,x)$. In particular, $F_w$ is $S$-crossing symmetric if and only if $w$ is real.
		
		To prove this, let $\varphi_1,\varphi_2,\psi_1,\psi_2\in L^2(\Rl)$. Then
		\begin{align*}
			Q_S(F_w)(\varphi_1\ot\varphi_2,\psi_1\ot\psi_2)
			&=
			\int \overline{\varphi_2(x)}\psi_2(y)w(x,y)\overline{\varphi_1(y)}\psi_1(x)\,dx\,dy
			\\
			&=
			\langle \varphi_1\ot\varphi_2,\,F_{\hat w}(\psi_1\ot\psi_2)\rangle.
		\end{align*}
		Since $F_{\hat w}$ is bounded, this shows $F_w\in\CrossSet_S$, and $\Cross_S(F_w)=F_{\hat w}$, which coincides with ${F_w}^*=F_{\hat{\overline w}}$ if and only if $w$ is real.
		
		\item For $\dim\Hil=\infty$, the identity operator $1\in\B(\Hil\tp{2})$ is {\em not} $S$-crossable for any $S\in\CS(\Hil)$. In fact, thanks to the polar decomposition and spectral theory, for any closed antilinear involution $S$ there exists an infinite set of orthonormal vectors $\{e_n\}_{n\in \Nl}\subset \dom(S)$ such that $0<\|S e_n\|\leq 1$. Then, we have $\Psi:=\sum_{n=1}^\infty\frac1n\,\frac{S e_n}{\|S e_n\|}\ot e_n\in\Hil\ot\Hil$ because $\langle\frac{S e_n}{\|S e_n\|}\ot e_n,\frac{S e_m}{\|S e_m\|}\ot e_m\rangle=\delta_{nm}$, but the $S$-crossed version of $1$ would satisfy
		\begin{align*}
			\langle\Psi,\Cross_S(1)(e_1\ot S e_1)\rangle
			&=
			\sum_{n=1}^\infty\frac{1}{n\|S e_n\|}\langle e_n\ot S^*S e_1,e_n\ot e_1\rangle
			\\
			&\geq
			\sum_{n=1}^\infty\frac{\|Se_1\|^2}{n}
			=
			\infty.
		\end{align*}
		Hence $\Cross_S(1)$ does not exist in $\B(\Hil\tp{2})$.
		
		\item For $\dim\Hil<\infty$, all operators in $\CB(\Hil\ot\Hil)$ are $S$-crossable for any $S\in\CS(\Hil)$. In particular, the identity is crossable. To describe $\Cross_S(1)$, we consider an orthonormal basis $\{e_n\}_n$ of $\Hil$ and set
		\begin{align}\label{def:xiS}
			\xi_S := \sum_{n=1}^{\dim\Hil} e_n\ot Se_n.
		\end{align}
		
		Note that $\|\xi_S\|^2 = \Tr(\Delta)$. We consider $P_S := |\xi_S\rangle\langle\xi_S|$, a multiple of the orthogonal projection onto $\Cl\xi_S$, which acts according to
		\begin{align}\label{eq:PS}
			P_S(v\ot w)=\langle Sv,w\rangle\cdot\xi_S.
		\end{align}
		An explicit calculation shows
		\begin{align}\label{eq:CrossId}
			\Cross_S(1)=P_S, \qquad \Cross_S(P_S)=1.
		\end{align}
		In particular, $\xi_S$ depends only on $S$ and not on the orthonormal basis $\{e_n\}_n$. The operator $P_S$ and categorical analogues of it will appear in Section~\ref{sec:symmetry} and Section~\ref{sec:catcross}, respectively.
	\end{enumerate}
\end{example}

After these examples, we establish some basic properties of the crossing map and its domain. Further properties of $\CrossSet_S(\Hil\tp{2})$, in particular relating to the norm topology of $\CB(\Hil\tp{2})$, will be presented in Section \ref{sec:crossingSbdd}.

\begin{proposition}\label{prop:generalcrossing}
	Let $S\in\CS(\Hil)$.
	\begin{enumerate}
		\item The set $\CrossSet_S(\Hil\tp{2})$ of all $S$-crossable operators is a complex linear subspace of $\B(\Hil\tp{2})$ that is dense in the strong and weak operator topologies, and
		\begin{align}
			\Cross_S:\CrossSet_S(\Hil\tp{2})\to\B(\Hil\tp{2})
		\end{align}
		is an injective complex linear map.
		
		\item Consider $\B(\Hil\tp{2})$ with either the norm topology, the strong operator topology, or the weak operator topology. Then, $\Cross_S:\CrossSet_S(\Hil\tp{2})\to \B(\Hil\tp{2})$ is a closed linear map.
		
		\item The $S$-crossing map $\Cross_S$ is continuous with respect to either the norm topology, the strong operator topology, or the weak operator topology if and only if $\dim\Hil<\infty$. Moreover, $\CrossSet_S(\Hil\tp{2})$ is closed if and only if $\dim\Hil<\infty$, in which case $\CrossSet_S(\Hil\tp{2}) = \B(\Hil\tp{2})$.
	\end{enumerate}
\end{proposition}
\begin{proof}
	a) The linearity of $\CrossSet_S(\Hil\tp{2})$ and $\Cross_S$ are clear from the definition, and injectivity of the crossing map follows from $S$ and $S^*$ having dense range. To show the density claim, we consider a sequence $P_n$ of orthogonal projections of finite rank that commute with $S$, are subprojections of the spectral projections $E^\Delta([\frac1n,n])$ of $\Delta$, and converge strongly to $1$ strongly as $n\to\infty$. Then we can estimate for any $T\in\CB(\Hil\ot\Hil)$ and any $\varphi_1,\varphi_2,\psi_1,\psi_2\in\Hil$
	\begin{align*}
		|Q_S(P_n^{\ot2}T P_n^{\ot 2})(\varphi_1\ot\varphi_2,\psi_1\ot\psi_2)|
		&\leq
		\|T\|\|P_n\varphi_2\|\|S^*P_n\psi_2\|\|SP_n\varphi_1\|\|P_n\psi_1\|
		\\
		&\leq
		n\|T\|\|P_n\varphi_2\|\|P_n\psi_2\|\|P_n\varphi_1\|\|P_n\psi_1\|.
	\end{align*}
	Since $P_n$ has finite rank, this estimate readily implies $P_n^{\ot2}T P_n^{\ot 2}\in\CrossSet_S(\Hil\tp{2})$. As $P_n\to1$ strongly, we have $P_n^{\ot2}T P_n^{\ot 2}\to T$ strongly, which proves our claim.
	
	b) Let $(T_n)_n\subset \CrossSet_S$ be a net such that $T_n\to T\in \B(\Hil\tp{2})$ and $\Cross_S(T_n)\to \hat T\in \B(\Hil\tp{2})$ in one of the specified topologies. Taking $\varphi_1, \psi_1\in \dom(S)$ and $\varphi_2, \psi_2\in \dom(S^*)$, we have
	\begin{align*}
		\langle\varphi_1\ot\varphi_2,\hat T \psi_1\ot \psi_2\rangle
		&=
		\lim_{n}\ip{\varphi_1\ot\varphi_2 }{\Cross_S(T_n) \psi_1\ot \psi_2}\\
		&=
		\lim_{n}\ip{\varphi_2\ot S^*\psi_2 }{T_n (S \varphi_1\ot \psi_1)}\\
		&=
		\ip{\varphi_2\ot S^\ast \psi_2 }{T(S \varphi_1\ot \psi_1)}.
	\end{align*}
	Since $\hat T\in \B(\Hil\tp{2})$, we see that $T\in \CrossSet_S(\Hil\tp{2})$ and $\Cross_S(T)=\hat T$. Hence $\Cross_S$ is closed.
	
	c) Let $\dim\Hil=\infty$. It follows from $1\not\in\CrossSet_S$, the WOT/SOT-density of $\CrossSet_S$ in $\B(\Hil\tp{2})$ and the fact that $\Cross_S$ is closed that $\Cross_S$ is not continuous.
	
	For the norm topology, let $S=J\Delta^{1/2}$ be the polar decomposition of $S$ and let $\{e_n\}_{n\in\Nl}$ be an orthonormal basis of $\Hil$ lying in $\dom(S)\cap\dom(S^*)$. Similar to our arguments in a), define $T_n := (P_1\ot P_n)(\Delta^{1/2}\otimes 1)$, where $P_n$ is the orthonormal projection onto $\operatorname{span}\{e_j\}_{j=1}^n$ for fixed $n\in\Nl$. Then we have $$T_n (\varphi\ot \psi)=\ip{\Delta^{1/2}e_1}{\varphi}e_1\ot P_n \psi,$$
	and $T_n\in \B(\Hil\tp{2})$ is a finite rank operator with $\|T_n\|\leq \|\Delta^{1/2}e_1\|$. Moreover, for $\psi\in\dom(S^*)$,
	\begin{align}
		\Cross_S(T_n)(\varphi\ot \psi)
		&=
		\sum_{i,j=1}^\infty \ip{J e_i \ot  e_j}{\Cross_S(T_n)(\varphi\ot \psi)} J e_i \ot e_j\\
		&=\sum_{i,j=1}^\infty \ip{e_j \ot S^*\psi}{T_n (SJe_i\ot \varphi)} J e_i \ot e_j\\
		&=\sum_{i,j=1}^\infty \ip{e_j \ot S^*\psi}{\ip{\Delta^{1/2}e_1}{\Delta^{-1/2}e_i}e_1\ot P_n \varphi} J e_i \ot e_j\\
		&=\ip{SP_n\varphi}{\psi}J e_1 \ot e_1.
	\end{align}
	
	Now, one can easily see that $\Cross_S(T_n)\in \B(\Hil^{\ot2})$ because $P_n\Hil$ is a finite dimensional subspace of $\dom(S)$. In addition, $\|\Cross_S(T_n)\|$ coincides with the norm of $(\ip{SP_n e_i}{e_j})_{i,j\in \Nl}\in \ell^2_{\Nl^2}$ and, although $\|T_n\|$ is uniformly bounded in $n$, $\|\Cross_S(T_n)\|$ will be uniformly bounded in $n$ if and only if $(\ip{Se_i}{e_j})_{i,j\in \Nl}\in \ell^2_{\Nl^2}\Leftrightarrow \Delta^{1/2}\in \CI_2(\Hil)$, where $\CI_2(\Hil)$ stands for the Hilbert--Schmidt ideal in $\B(\Hil)$. Since the relation $J\Delta J=\Delta^{-1}$ implies that $\Delta^{1/2}\notin \CI_2(\Hil)$, unless $\dim\Hil<\infty$, we conclude that $\Cross_S(T)$ must be unbounded.
	
	Finally, it follows from the Closed Graph Theorem that $\CrossSet_S$ cannot be closed in norm, hence neither in the strong nor weak operator topologies.
\end{proof}

\begin{remark}
	Let $S\in\CS(\Hil)$. In general, the subspace $\CrossSet_S(\Hil\tp{2})$ is neither a subalgebra of $\B(\Hil\tp{2})$, nor $*$-stable: Indeed, the flip $F$ satisfies $F^2 = 1$ and~$F$ is always $S$-crossable, while $1$ is $S$-crossable if and only if $\dim\Hil < \infty$ by Example~ ~\ref{example:basics} a). If $S$ is not bounded, one can use Example \ref{example:basics} b) to show that $\CrossSet_S(\Hil\tp{2})$ is not $*$-invariant: Choose $A,B\in \B(\Hil)$ such that $S^\ast B^\ast S^\ast\in \B(\Hil)$ and $S^\ast A S^\ast\not\in \B(\Hil)$. Then $T:=(A\ot B)F \in \CrossSet_S(\Hil\ot\Hil)$ and $T^* \notin \CrossSet_S(\Hil\ot\Hil)$.
\end{remark}

Given $S\in\CS(\Hil)$ and its polar decomposition $S=J\Delta^{1/2}$, also $J\in\CS(\Hil)$ is an antilinear (and antiunitary) involution. We now consider the relation between the crossing maps $\Cross_S$ and $\Cross_J$, and their link to the KMS condition from Tomita--Takesaki modular theory. This connection to the KMS condition was the primary use of crossing symmetry in the context of twisted Araki--Woods von Neumann algebras \cite{CorreaDaSilvaLechner:2023}, and will be used in the present paper in Section~\ref{sec:endos}.

\begin{proposition}\label{prop:KMS-likeCrossing}
	Let $S\in \CS(\Hil)$ with $S=J\Delta^{1/2}$ its polar decomposition, $T\in\CrossSet_S$, and $\psi_1,\ldots,\psi_4\in\Hil$. Then the function $T^{\psi_1 \psi_2}_{\psi_3 \psi_4}:\Rl\to\Cl$, defined by
	\begin{align}
		T^{\psi_1 \psi_2}_{\psi_3 \psi_4}(t):=\ip{\psi_1 \ot \psi_2}{T(t)\psi_3 \ot \psi_4}
		,\qquad
		T(t):=(1\ot \Delta^{-it})T(\Delta^{it}\ot 1 ),
	\end{align}
	has a bounded analytic extension to the strip $\Strip_{1/2} := \{z\in\Cl\,:\,0<\Im(z)<\frac{1}{2}\}$ satisfying the boundary condition, $t\in\Rl$,
	\begin{align}\label{eq:bndcond}
		T^{\psi_1 \psi_2}_{\psi_3 \psi_4}\left(t+\tfrac{i}{2}\right)
		=
		\ip{J\psi_3 \ot \psi_1}{(\Delta^{-it}\ot 1)\Cross_S(T)(1\ot \Delta^{it})(\psi_4 \ot J\psi_2)}.
	\end{align}
\end{proposition}

\begin{proof}
	Let $\psi_1,\ldots,\psi_4\in\Hil$ be vectors, initially with $\psi_2,\psi_3$ entire analytic for $t\mapsto\Delta^{it}$.
	Then $T^{\psi_1 \psi_2}_{\psi_3 \psi_4}$ has an entire extension, and for $t\in\Rl$ we have
	\begin{align*}
		T^{\psi_1 \psi_2}_{\psi_3 \psi_4}\left(t+\tfrac{i}{2}\right)
		&=
		\ip{\psi_1 \ot \Delta^{1/2}\Delta^{it}\psi_2}{T(\Delta^{-1/2}\Delta^{it}\psi_3  \ot \psi_4)}
		\\
		&=\ip{\psi_1 \ot S^*J\Delta^{it}\psi_2}{T(SJ\Delta^{it}\psi_3  \ot \psi_4)}
		\\
		&=\ip{J\Delta^{it}\psi_3 \ot \psi_1}{\Cross_S(T)(\psi_4 \ot J\Delta^{it}\psi_2)}
		\\
		&=\ip{J\psi_3 \ot \psi_1}{(\Delta^{-it}\ot 1)\Cross_S(T)(1\ot \Delta^{it})(\psi_4 \ot J\psi_2)}
	\end{align*}
	which is the required boundary condition \eqref{eq:bndcond}. To remove the domain assumption on $\psi_2$ and $\psi_3$, we note
	\begin{align*}
		\sup_{t\in\Rl}\left\{\left\|T^{\psi_1 \psi_2}_{\psi_3 \psi_4}\left(t+\tfrac{i}{2}\right)\right\|,	\left\|T^{\psi_1 \psi_2}_{\psi_3 \psi_4}(t)\right\|\right\}\leq (\|T\|+\|\Cross_S(T)\|) \|\psi_1\|\|\psi_2\|\|\psi_3\|\|\psi_4\|.
	\end{align*}
	Therefore, if $\psi_2,\psi_3 \in \Hil$ are two arbitrary vectors and $(\psi_{2,n})_n,(\psi_{3,n})_n\subset \Hil$ are sequences of analytic vectors for $\Delta$ with $\psi_{2,n}\to \psi_2$ and $\psi_{3,m}\to \psi_3$, then
	$(T^{\psi_1 \psi_{2,n}}_{\psi_{3,m} \psi_4})_n$ is a uniform Cauchy sequence (in $n\in \Nl$, for every fixed $m\in\Nl$) of analytic functions in the interior of the strip~$\Strip_{1/2}$. Therefore, its limit $T^{\psi_1 \psi_2}_{\psi_{3,m} \psi_4}$ is analytic in the interior of the strip and it is easy to check that it satisfies \eqref{eq:bndcond}. The same argument applied to the sequence $(T^{\psi_1 \psi_2}_{\psi_{3,m} \psi_4})_m$ proves the claim.
\end{proof}

Conversely, if $T\in\B(\Hil\tp{2})$ satisfies the analyticity and boundary conditions in the thesis of Proposition \ref{prop:KMS-likeCrossing}, we can relate $\Cross_S(T)$ with $\Cross_J(T)$. Namely, for vectors $\psi_1,\ldots,\psi_4\in\Hil$ with $\psi_2,\psi_3$ entire analytic for $t\mapsto\Delta^{it}$, we can write the analytic extension of $T^{\psi_1,\psi_2}_{\psi_3,\psi_4}(t+\frac{i}{2})$ to the point~$-\frac{i}{2}$ in two different ways using \eqref{eq:bndcond}:
\begin{align*}
	T^{\psi_2,J\psi_4}_{J\psi_1,\psi_3}\left(-\tfrac{i}{2}+\tfrac{i}{2}\right)
	&=\ip{\psi_1 \ot \psi_2}{(\Delta^{-1/2}\ot 1)\Cross_S(T)(1\ot \Delta^{1/2})(\psi_3 \ot \psi_4)}\\
	&=T^{\psi_2,J\psi_4}_{J\psi_1,\psi_3}(0)\\
	&=\ip{\psi_2 \ot J\psi_4}{T(J\psi_1 \ot \psi_3)}\\
	&=\ip{\psi_1 \ot \psi_2}{ \Cross_J(T)(\psi_3 \ot \psi_4)}.
\end{align*}
That means that, whenever $T\in \CrossSet_J\cap \CrossSet_S$, we have
$$(\Delta^{-1/2}\ot 1)\Cross_S(T)(1\ot \Delta^{1/2})\subset \Cross_J(T)$$
or, equivalently,
$$(\Delta^{1/2}\ot 1)\Cross_J(T)(1\ot \Delta^{-1/2})\subset \Cross_S(T).$$
In particular, if $T\in \CrossSet_J\cap \CrossSet_S$, then the operators $(\Delta^{1/2}\ot 1)\Cross_J(T)(1\ot \Delta^{-1/2})$ and $(\Delta^{-1/2}\ot 1)\Cross_S(T)(1\ot \Delta^{1/2})$ are bounded. The next proposition contains this observation and its converse.

\begin{proposition}\label{prop:CrossableAndBounded}
	Let $S=J\Delta^{1/2}\in\CS(\Hil)$.
	\begin{enumerate}
		\item If $T\in \CrossSet_J\cap \CrossSet_S$, then the operators
		$(\Delta^{1/2}\ot 1)\Cross_J(T)(1\ot \Delta^{-1/2})$ and $(\Delta^{-1/2}\ot 1)\Cross_S(T)(1\ot \Delta^{1/2})$ are bounded.
		\item If $T\in\CrossSet_S$ and $\|(1\ot\Delta^{1/2})T(\Delta^{-1/2}\ot1)\|<\infty$, then $(1\ot\Delta^{1/2})T(\Delta^{-1/2}\ot1)$ is $J$-crossable, and $$\Cross_J((1\ot\Delta^{1/2})T(\Delta^{-1/2}\ot1))=\Cross_S(T).$$
		\item If $T\in\CrossSet_J$ and $\|(\Delta^{1/2}\ot1)\Cross_J(T)(1\ot\Delta^{-1/2})\|<\infty$, then $T\in\CrossSet_S$ and $$(\Delta^{1/2}\ot 1)\Cross_J(T)(1\ot \Delta^{-1/2})=\Cross_S(T).$$
	\end{enumerate}
\end{proposition}

\begin{proof} The first statement has already been discussed as a consequence of Proposition \ref{prop:KMS-likeCrossing}. The remaining statements follow from two straightforward calculations. For every $\psi_1, \varphi_1\in \dom(S)$ and $\psi_2, \varphi_2\in \dom(S^*)$, we have for $b)$
	\begin{align*}
		\ip{\psi_1\ot \psi_2}{\Cross_S(T)(\varphi_1\ot \varphi_2)}
		&=
		\langle\psi_2 \ot S^*\varphi_2,T (S\psi_1\ot \varphi_1)\rangle
		\\
		&=
		\langle\psi_2\ot J\varphi_2,(1\ot \Delta^{1/2})T(\Delta^{-1/2}\ot 1)(J\psi_1\ot \varphi_1)\rangle
		\\
		&=
		\langle\psi_1\ot \psi_2,\Cross_J( (1\ot \Delta^{1/2})T(\Delta^{-1/2}\ot 1)) (\varphi_1\ot \varphi_2)\rangle
	\end{align*}
	and for $c)$ we have
	\begin{align*}
		\langle\psi_1\ot \psi_2,(\Delta^{1/2} \ot 1)&\Cross_J(T)(1\ot \Delta^{-1/2})(\varphi_1\ot \varphi_2)\rangle \\
		&=
		\langle\Delta^{1/2}\psi_1\ot \psi_2,\Cross_J(T)(\varphi_1\ot \Delta^{-1/2}\varphi_2)\rangle
		\\
		&=
		\langle\psi_2\ot J\Delta^{-1/2}\varphi_2,T(J\Delta^{1/2}\psi_1\ot \varphi_1)\rangle\\
		&=
		\ip{\psi_1\ot \psi_2}{\Cross_S(T)(\varphi_1\ot \varphi_2)}.
	\end{align*}
\end{proof}

In general, one cannot improve the results of Proposition \ref{prop:CrossableAndBounded} without additional assumptions, namely, there are examples even of $J$-crossing symmetric operators which are not $S$-crossable for unbounded $S$. For bounded $S$, instead, we shall see in the next section in particular that $\CrossSet_S = \CrossSet_J$.

\section{The crossing map for bounded involutions}\label{sec:crossingSbdd}

In the previous section we have established the main features of the crossing map that hold for general $S\in\CS(\Hil)$. In case $S=J\Delta^{1/2}$ is bounded (i.e.\ $\Delta$ is bounded) or even antiunitary (i.e.\ $\Delta=1$), several questions simplify and new features appear. In this section we therefore concentrate on the case of a bounded antilinear involution $S\in\CS(\Hil)$. In case $\Delta=1$ and $\dim\Hil<\infty$, a related analysis was carried out by Peschik \cite{Peschik:2023}.

The first observation is that for bounded $S$, the domain $\CrossSet_S(\Hil \otimes \Hil)$ of the crossing map is independent of $S$. In fact, if $S_1,S_2$ are two bounded antilinear involutions on $\Hil$, a direct calculation shows
\begin{align}\label{eq:relationS1S2}
	\CrossSet_{S_1}(\Hil\ot\Hil)
	&=
	\CrossSet_{S_2}(\Hil\ot\Hil),
	\\
	\Cross_{S_1}(T)
	&=
	(S_1^*S_2^* \ot1)\Cross_{S_2}(T)(1\ot S_2^*S_1^*).
\end{align}
We will therefore write $\CrossSetb(\Hil\tp{2})$ or just $\CrossSetb$ instead of $\CrossSet_S(\Hil\tp{2})$ for bounded $S$.

The domain $\CrossSetb$ can be characterized as follows.

\begin{lemma}\label{lemma:crossability}
	Let $S\in\CS(\Hil)$ be bounded and $T\in\B(\Hil\tp{2})$. Then the following statements are equivalent:
	\begin{enumerate}
		\item $T\in\CrossSetb$.
		\item For any $\psi_1,\psi_2\in\Hil$, the map $A_{\psi_1\ot\psi_2}:\Hil\odot\Hil\to\Cl$,
		\begin{align}
			A_{\psi_1\ot\psi_2}:
			\varphi_1\ot\varphi_2
			\mapsto
			\langle\varphi_2\ot S^*\psi_2,T(S\varphi_1\ot\psi_1)\rangle
		\end{align}
		extends continuously and bi-antilinearly to $\Hil\ot\Hil$, and, for any $\Phi\in\Hil\ot\Hil$, the map $\psi_1\ot\psi_2\mapsto A_{\psi_1\ot\psi_2}(\Phi)$ extends continuously and bilinearly to $\Hil\ot\Hil$.
		\item There exists $c>0$ such that for all $a,b\in\ell^2_{\Nl^2}$,
		\begin{align}\label{eq:Cdom1}
			\left|\sum_{i,j,k,l=1}^\infty a_{ij}\,T^{jl}_{ik}\,b_{kl}\right|
			\leq
			c\|a\|_{\ell^2_{\Nl^2}}\|b\|_{\ell^2_{\Nl^2}},
		\end{align}
		where $T^{jl}_{ik}:=\langle e_j\ot e_l,T(e_i\ot e_k)\rangle$ for some arbitrary orthonormal basis $\{e_n\}_n$ of $\Hil$.
	\end{enumerate}
	In particular, 
	if $T\in\CrossSetb$, then $T^*\in\CrossSetb$ by observing that $(T^*)^{jl}_{ik} = \overline{T_{jl}^{ik}}$.
	
	Moreover, if $S$ is antiunitary (i.e.\ $S=J$) and the orthonormal basis is chosen such that $Se_k=e_k$,
	we have the formula $\Cross_S(T)^{ij}_{kl} = T^{jl}_{ik}$.
\end{lemma}
\begin{proof}
	According to our previous remarks, we may assume that $S$ is antiunitary without loss of generality.
	
	The continuous extension properties in $b)$ exactly specify that the quadratic form $Q_S(T)$ is given by a bounded operator, i.e.\ $a)\Leftrightarrow b)$.
	
	The equivalence $a)\Leftrightarrow c)$ follows by observing that it is independent of the orthonormal basis $\{e_n\}_n$ chosen, taking $S$ to be the antiunitary involution fixed by $Se_n=e_n$, and by the fact that antiunitary operators map orthonormal bases to orthonormal bases.
\end{proof}

As the identity is never crossable in infinite dimensions (Example~\ref{example:basics}~d)), we know in particular $\CrossSetb\neq\CB(\Hil\tp{2})$. The following statement strengthens this.

\begin{corollary}\label{cor:notdense}
	Let $S\in\CS(\Hil)$ be bounded and $\dim\Hil=\infty$. Then
	\begin{align}
		\inf_{T\in\CrossSetb}\|T-1\|=1.
	\end{align}
	In particular, $\CrossSetb$ is not norm dense in $\B(\Hil^{\ot 2})$.
\end{corollary}
\begin{proof}
	Let $T\in\CB(\Hil\tp{2})$ with $\|T-1\|=:\eps<1$. We have to show $T\not\in\CrossSetb$.
	
	For any orthonormal basis $\{e_n\}_{n\in \Nl}$ we have $|\ip{e_i\ot e_j}{(T-1)e_k\ot e_l}|<\eps$ for all $i,j,k,l\in \Nl$ which translates as $|T^{ij}_{kl}-\delta_{ik}\delta_{jl}|<\eps$ for all $i,j,k,l\in \Nl$. In particular,  $\delta_{ik}\delta_{jl}-\eps<\Re T^{ij}_{kl}<\delta_{ik}\delta_{jl}+\eps$. Let now $(a_{ij})_{i,j\in \Nl}\in \ell^2_{\Nl^2}$ be a any positive sequence with finitely many non-vanishing terms such that $a_{ij}=0$ if $i\neq j$, and let $b_{kl} := \delta_{1k}\delta_{1l}$. Then
	\begin{align}
		\Re\left(\sum_{i,j,k,l\in \Nl} a_{ij}T^{jl}_{ik} b_{kl}\right)&\geq \sum_{i,j,k,l\in \Nl} a_{ij}(\delta_{ij}\delta_{kl}-\eps) b_{kl}\\
		&=(1-\eps)\sum_{i\in \Nl} a_{ii}.
	\end{align}
	If $T\in\CrossSetb$, then the expression on the left is bounded above in modulus by $c\|a\|_{\ell^2_{\Nl^2}}$ (Lemma~\ref{lemma:crossability}) for some constant $c>0$. But when we vary $(a_{ii})$ within the above limitations such that $\|a\|_{\ell^2_{\Nl^2}}=1$, we can make $\sum_ia_{ii}$ as large as we like. Hence $T$ is not crossable. As $0\in\CrossSetb$, the proof is finished.
\end{proof}

\begin{corollary}
	Let $S\in\CS(\Hil)$ be bounded. Then, the space $\B_{0}(\Hil\tp{2})$ of finite rank operators is a core for $\Cross_S$ w.r.t.\ the strong and weak operator topologies.
\end{corollary}
\begin{proof}
	Let $S=J\Delta^{1/2}$ be the polar decomposition of $S$. Let also $T\in \CrossSetb$ and let $\{e_i\}_{i\in \Nl}$ be an orthonormal basis of $\Hil$ with $Je_i=e_i$. Define the finite rank operator $T_n\in \B_0(\Hil\tp{2})$, $n\in\Nl$, by its coefficients in the basis $\{e_i\ot e_j\}_{i,j\in \Nl}$, namely $(T_n)^{ij}_{kl}=:T^{ij}_{kl}$ if $i,j,k,l\leq n$ and $(T_n)^{ij}_{kl}:=0$ otherwise. Moreover, $T_n\in \CrossSetb$ because it is finite rank.
	
	Since $(T-T_n)^{ij}_{kl}=0$ for $i,j,k,l\leq n$, we have $T_n\to T$ either in SOT or WOT. Similarly, since $\Cross_J(T-T_n)^{ij}_{kl}=(T-T_n)^{jl}_{ik}=0$ for $i,j,k,l\leq n$, we also have ${\Cross_J(T_n)\to \Cross_J(T)}$ in SOT or WOT. Using now equation \eqref{eq:relationS1S2}, in the same topologies, $\Cross_S(T_n)=(S^*J \ot1)\Cross_{J}(T_n)(1\ot JS^*)\to (S^*J \ot1)\Cross_J(T)(1\ot JS^*)$. This proves the claimed core property.
\end{proof}

The arguments in this proof show in particular that $\CrossSetb$ is dense in $\B(\Hil\tp{2})$ in the strong operator topology. We already know, in the case $\dim\Hil=\infty$, that $\Cross_S$ is not SOT- or WOT-continuous. Nonetheless, there exist interesting subspaces of the domain $\CrossSetb$ on which the crossing map acts isometrically w.r.t.\ suitable norms, as we show next.

Firstly, we consider the completed projective Banach space tensor product $\ot_\pi$ and the subspace
\begin{align}
	\CD_\pi
	&:=
	\{XF\,:\, X\in\B(\Hil)\ot_\pi\B(\Hil)\} =\B(\Hil)^{\ot_\pi2}F,
\end{align}
of $\B(\Hil\tp{2})$. We write $\|\cdot\|_\pi$ for the projective tensor norm on $\B(\Hil)^{\ot_\pi2}$. Secondly, we consider the Hilbert--Schmidt ideal $\CI_2:=\CI_2(\Hil\tp{2})\subset\B(\Hil\tp{2})$ and write $\|\cdot\|_2$ for its Hilbert--Schmidt norm.

\begin{proposition}\label{prop:Cross-isometric}
	Let $J$ be an antiunitary involution on $\Hil$.
	$\CD_\pi\subset\CrossSet$, and $\Cross_J$ restricts to a bijective linear map $\CD_\pi\to\CD_\pi$, $XF\mapsto \hat XF$, with $\|\hat X\|_\pi=\|X\|_\pi$.
\end{proposition}

\begin{proof}
	Recall that any element $X\in\B(\Hil)\ot_\pi\B(\Hil)$ can be represented in the form $X=\sum_{n=1}^\infty A_n\ot B_n$ with $A_n,B_n\in\B(\Hil)$ and $\sum_{n=1}^\infty\|A_n\|\|B_n\|<\infty$, and $\|X\|_\pi$ is defined as the infimum of $\sum_{n=1}^\infty\|A_n\|\|B_n\|$ over all such representations. As the $J$-transposition $B\mapsto JB^*J$ is isometric on $\B(\Hil)$, also $\hat X:=\sum_{n=1}^\infty JB_n^*J\ot A_n$ lies in $\B(\Hil)\ot_\pi\B(\Hil)$, with $\|\hat X\|_\pi=\|X\|_\pi$.
	
	Similar to Example~\ref{example:basics}~b), we find for $\psi_1,\psi_2, \varphi_1,\varphi_2\in \Hil$
	\begin{align*}
		\langle\varphi_2\ot J\psi_2,XF(J\varphi_1\ot\psi_1)\rangle
		&=
		\sum_{n=1}^\infty \langle\varphi_2,A_n\psi_1\rangle\langle J\psi_2,B_nJ\varphi_1\rangle
		\\
		&=
		\sum_{n=1}^\infty \langle\varphi_1\ot\varphi_2,(JB_n^*J\ot A_n)F(\psi_1\ot\psi_2)\rangle
		\\
		&=
		\langle\varphi_1\ot\varphi_2,\hat XF(\psi_1\ot\psi_2)\rangle.
	\end{align*}
	This shows $XF\in\CrossSet$ and $\Cross(XF)=\hat XF$, hence $\Cross(\CD_\pi)\subset\CD_\pi$. It is clear from Proposition~\ref{prop:bij} that the restriction is bijective.
\end{proof}

We now study invariance properties of the domain $\CrossSetb$ that are particular to bounded $S$ and can be expressed in terms of a Banach space analogue of standard subspaces. To that end, we fix a bounded $S\in\CS(\Hil)$ with polar decomposition $S=J\Delta^{1/2}$ and introduce the (anti-)linear maps
\begin{align}
	\Deltatinout_\alpha,\Deltat_\alpha,\JF,\St&:\CB(\Hil\tp{2})\to\CB(\Hil\tp{2}),\\
	\Deltatin_\alpha(T)
	&:=
	(1\ot\Delta^\alpha)T(\Delta^{-\alpha}\ot1),
	\\
	\Deltatout_\alpha(T)
	&:=
	(\Delta^\alpha\ot1)T(1\ot\Delta^{-\alpha}),
	\\
	\Deltat_\alpha
	&:=
	\Deltatin_\alpha\Deltatout_\alpha,\\
	\JF(T)
	&:=
	F(J\ot J)T(J\ot J)F,
	\\
	\St(T) &:= T^*.
\end{align}
All these maps are bounded in operator norm. In the case of $\Deltatinout_\alpha$, this is a consequence of the symmetry of the spectrum of $\Delta$ which follows from the modular relation $J\Delta=\Delta^{-1}J$ (i.e.\ $\Delta$ is bounded if and only if $\Delta^{-1}$ is bounded).

Moreover, it is clear that $\JF,\Deltatinout_\alpha,\Deltat_\alpha,\St$ are bijections of $\CB(\Hil\tp{2})$, with inverses $\JF^{-1}=\JF$, $(\Deltatinout_\alpha)^{-1}=\Deltatinout_{-\alpha}$, $\Deltat_\alpha^{-1}=\Deltat_{-\alpha}$, and $\St^{-1}=\St$.

\begin{proposition}\label{prop:bij}
	Let $S\in\CS(\Hil)$ be bounded.
	\begin{enumerate}
		\item The crossing map $\Cross_S:\CrossSetb\to\CrossSetb$ is a bijection of its domain. Also the maps $\St,\JF$, and $\Deltatinout_\alpha$, $\Deltat_\alpha$, $\alpha\in\Rl$, restrict to bijections of $\CrossSetb$.
		
		\item As maps $\CrossSetb\to\CrossSetb$, one has
		\begin{align}
			\Cross_S^2
			&=
			\JF\Deltat_{-1/2}\St,
			\\
			\Cross_S\Deltatinout_{\alpha}
			&=
			\Deltatoutin_{\alpha}\Cross_S,
			\\
			\Cross_S\JF
			&=
			\Deltat_{1/2}\St\Cross_S,
			\\
			\Cross_S\St
			&=
			\JF\Deltat_{-1/2}\Cross_S.
		\end{align}
		
		\item The fourth power of the crossing map is
		\begin{align}
			\Cross_S^4(T)=(\Delta\ot\Delta)T(\Delta^{-1}\ot\Delta^{-1}),\qquad T\in\CrossSetb.
		\end{align}
		In particular, the crossing map is of order four in case $S$ is antiunitary.
		
		\item The inverse crossing map is given by
		\begin{align}
			\Cross_S^{-1}(T)
			&=
			\Cross_S(T^*)^*,
		\end{align}
		and is uniquely fixed by the inverse crossing relation (for all $\varphi_1,\varphi_2,\psi_1,\psi_2$)
		\begin{align}
			\langle\varphi_1\ot\psi_1,\Cross_S^{-1}(T)(\varphi_2\ot\psi_2)\rangle
			=
			\langle S\varphi_2\ot\varphi_1,T(\psi_2\ot S^*\psi_1)\rangle
			.
		\end{align}
	\end{enumerate}
\end{proposition}

\begin{proof}
	We begin with some calculations based on an arbitrary $T\in\CrossSetb$ and arbitrary vectors $\psi_1,\psi_2,\varphi_1,\varphi_2\in\Hil$. By repeatedly using the definition of the crossing map and the modular exchange relation $J\Delta^{1/2}=\Delta^{-1/2}J$, we find
	\begin{align*}
		Q_S(\Cross_S(T))(\varphi_1\ot\varphi_2,\psi_1\ot\psi_2)
		&=
		\ip{\varphi_2\ot S^*\psi_2}{\Cross_S(T)(S\varphi_1\ot \psi_1)}
		\\
		&=
		\ip{S^*\psi_2\ot S^*\psi_1 }{T (S\varphi_2 \ot S \varphi_1)}
		\\
		&=\ip{\varphi_1\ot\varphi_2}{F(S^*)^{\ot 2}T^*(S^*)^{\ot 2}F (\psi_1\ot \psi_2)},
		\\
		Q_S(\Deltatin_\alpha(T))(\varphi_1\ot\varphi_2,\psi_1\ot\psi_2)
		&=
		\langle \varphi_2\ot S^*\Delta^{-\alpha}\psi_2,T(S\Delta^\alpha\varphi_1\ot \psi_1)\rangle
		\\
		&=
		\langle\varphi_1\ot\varphi_2,\Deltatout_\alpha(\Cross_S(T))(\psi_1\ot\psi_2)\rangle,
	\end{align*}
	\begin{align*}
		Q_S(\JF(T))(\varphi_1\ot\varphi_2,\psi_1\ot\psi_2)
		&=
		\langle \Delta^{-1/2}\psi_2\ot J\varphi_2,T(J\psi_1\ot\Delta^{1/2}\varphi_1)\rangle^*
		\\
		&=
		\langle SJ\psi_1\ot\Delta^{-1/2}\psi_2,\Cross_S(T)(\Delta^{1/2}\varphi_1\ot S^*J\varphi_2)\rangle^*
		\\
		&=
		\langle\varphi_1\ot\varphi_2,\Deltah_{1/2}(\Cross_S(T)^*)(\psi_1\ot\psi_2)\rangle,
		\\
		Q_S(T^*)(\varphi_1\ot\varphi_2,\psi_1\ot\psi_2)
		&=
		\langle S\varphi_1\ot\psi_1,T(\varphi_2\ot S^*\psi_2)\rangle^*
		\\
		&=
		\langle\varphi_1\ot\varphi_2,F(S^*\ot S^*)\Cross_S(T)(S^*\ot S^*)F(\psi_1\ot\psi_2)\rangle.
	\end{align*}
	As all occurring operators are bounded, this shows that $\Cross_S,\JF,\Deltatinout_{\alpha}$, $\Deltat_\alpha$ and $\St$ map $\CrossSetb$ into $\CrossSetb$. (The ``out'' version of the second equation follows by an analogous calculation which we omit.) Taking also into account $F(S^*)^{\ot 2}T(S^*)^{\ot 2}F=\JF\Deltat_{-1/2}$, we can read off the four equations listed  in claim b), for now to be understood as equations between maps $\CrossSetb\to\CrossSetb \subset\CB(\Hil\tp{2})$.
	
	Since $\JF$ and $\St$ are involutions and $(\Deltatinout_{\alpha})^{-1}=\Deltatinout_{-\alpha}$, it follows that $\JF$, $\St$, and $\Deltatinout_\alpha$, $\Deltat_\alpha$ restrict to bijections of $\CrossSetb$. The first exchange relation, $\Cross_S^2=\JF\Deltat_{-1/2}\St$, then shows that $\Cross_S:\CrossSetb\to\CrossSetb$ is surjective and hence bijective, as injectivity was proven in Proposition \ref{prop:generalcrossing} $a)$. This proves $a$) and~$b)$.
	
	For $c)$, we first note the three exchange relations $\St\JF=\JF\St$, $\St\Deltat_{\alpha}=\Deltat_{-\alpha}\St$, $\JF\Deltat_{\alpha}=\Deltat_{-\alpha}\JF$, all of which are verified immediately. We then calculate with the help of the first equation from $b)$
	\begin{align*}
		\Cross_S^4
		&=
		\JF\Deltat_{-1/2}\St\JF\Deltat_{-1/2}\St
		=
		\JF\Deltat_{-1/2}\JF\Deltat_{1/2}
		=
		\Deltat_1,
	\end{align*}
	which is the claimed equation for the fourth power of the crossing map.
	
	For $d)$, we calculate with the equations from $b)$
	\begin{align}
		\St\Cross_S\St\Cross_S
		&=
		\St\JF\Deltat_{-1/2}\Cross_S^2
		=
		\St\JF\Deltat_{-1/2}\JF\Deltat_{-1/2}\St
		=1.
	\end{align}
	This implies $\Cross_S^{-1}=\St\Cross_S\St$, as claimed. The other characterization of the inverse is easily verified by direct calculation.
\end{proof}

As a corollary to this result, we can now describe crossing symmetry in terms of a Banach space analogue of standard subspaces. To make the connection to the Hilbert space situation most apparent, we define
\begin{align}
	\CX
	:=
	\overline{\CrossSetb}^{\|\cdot\|}
\end{align}
as the operator norm closed subspace of $\CB(\Hil\tp{2})$ generated by the crossable operators. Note $\CX\neq\CB(\Hil\tp{2})$ unless $\dim\Hil < \infty$ (Corollary \ref{cor:notdense}). We define
\begin{align*}
	\Sh,\Jh&:\CrossSetb\to\CrossSetb
	\\
	\Sh
	&:=\St\Cross_S : T\mapsto \Cross_S(T)^*,
	\\
	\Jh
	&:=
	\St\Cross_J:T\mapsto \Cross_J(T)^*,
	\\
	\Deltah_{\alpha}
	&:=
	\Deltatin_{\alpha}.
\end{align*}

\begin{theorem}\label{thm:Hhat}
	Let $S\in\CS(\Hil)$ be bounded. Then the maps $\Sh$ and $\Jh$ are bounded antilinear involutions on $\CrossSetb$ satisfying
	\begin{align}
		\Sh=\Jh\Deltah_{1/2}=\Deltah_{-1/2}\Jh.
	\end{align}
	An operator $T\in\CrossSetb$ is $S$-crossing symmetric if and only if $T\in\Hh$, with
	\begin{align}
		\Hh
		:=\ker(\Sh-1).
	\end{align}
	The space $\Hh$ is a closed real subspace of $\CX = \overline{\CrossSetb}^{\|\cdot\|}$ which is also cyclic and separating, i.e.\ $\Hh+i\Hh\subset\CX$ is dense in operator norm and $\Hh\cap i\Hh=\{0\}$.
\end{theorem}

\begin{proof}
	By the preceding proposition we know that $\Sh$, $\Jh$, and $\Deltah_{1/2}$ are bijections of $\CrossSetb$, with $\Sh$, $\Jh$ antilinear and $\Deltah_{1/2}$ linear. Using relation $d)$ of Proposition \ref{prop:bij}, we see that $\Sh$ is an involution,
	\begin{align}
		\Sh^2
		&:=
		\St\Cross_S\St\Cross_S=\Cross_S^{-1}\Cross_S=1.
	\end{align}
	The same reasoning applied to $J$ instead of $S$ shows that $\Jh$ is an involution as well.
	
	To show $\Sh=\Jh\Deltah_{1/2}=\Deltah_{-1/2}\Jh$, observe that \eqref{eq:relationS1S2} implies $\Cross_S = \Cross_J \Deltatin_{1/2} = \Deltatout_{1/2} \Cross_J$.
	Hence, by Proposition \ref{prop:bij} $b)$,
	\begin{align*}
		\Sh
		&=
		\St\Cross_S
		=
		\St\Deltatout_{-1/2}\Cross_S\Deltatin_{1/2}
		=
		\St\Cross_J\Deltatin_{1/2}
		=
		\Jh\Deltah_{1/2}
		\\
		&=
		\St\Deltatout_{1/2}\Cross_J
		=
		\Deltatin_{-1/2}\St\Cross_J
		=
		\Deltah_{-1/2}\Jh.
	\end{align*}
	
	By Definition \ref{def:Scrossable}, $T\in\CrossSetb$ is $S$-crossing symmetric if and only if $\Cross_S(T)=T^*$, i.e.\ if and only if $T\in\Hh$. Moreover, $\Hh$ is a norm closed real linear space as the kernel of a bounded real linear operator. The statements about $\Hh$ being cyclic and separating are immediate consequences of the definition of $\CX$ and the fact that $\Sh$ is an antilinear involution.
\end{proof}

In the present Banach space setting, we have no polar decomposition of $\Sh$ at our disposal. To demonstrate that the above result is really a close analogue of standard subspaces of Hilbert spaces, we next consider the Hilbert space $\CI_2\subset\CB(\Hil\tp{2})$ of Hilbert--Schmidt operators on $\Hil\ot\Hil$. We write $\|\cdot\|_2$ and $\langle\,\cdot\,,\,\cdot\,\rangle_2$ for the Hilbert--Schmidt norm and scalar product, respectively.

\begin{proposition}\label{prop:HS-Crossing}
	\leavevmode
	\begin{enumerate}
		\item Any Hilbert--Schmidt operator is $S$-crossable for any bounded $S\in\CS(\Hil)$, i.e.\ $\CI_2\subset\CrossSetb$. The $J$-crossing map $\Cross_J$ restricts to a unitary on $\CI_2$, and the $S$-crossing map $\Cross_S$ restricts to a bounded bijection on $\CI_2$.
		\item In restriction to $\CI_2\subset\CrossSetb$,
		\begin{align}
			\Sh|_{\CI_2}
			&=
			\Jh|_{\CI_2} {\Deltah_{1/2}}|_{\CI_2}
		\end{align}
		is the polar decomposition of $\Sh|_{\CI_2}$. Hence $\Hh\cap\CI_2$ is a standard subspace of the Hilbert space $\CI_2$, with Tomita operator
		\begin{align}
			\Sh_{\Hh\cap\CI_2} = \Sh|_{\CI_2}.
		\end{align}
	\end{enumerate}
\end{proposition}

\begin{proof}
	$a)$ Using $\Cross_J=\Deltatout_{-1/2}\Cross_S$ as in the previous proof, and the fact that $\Deltatout_{-1/2}$ preserves $\CI_2\subset\CrossSetb$ thanks to the ideal properties of Hilbert--Schmidt operators, we only have to show that $\Cross_J$ preserves $\CI_2$.
	
	Let $\{e_n\}_{n\in \Nl}$ be an orthonormal basis of $\Hil$. Then $\{e_n\ot e_m\}_{n,m\in \Nl}$, $\{e_n\ot J e_m\}_{n,m\in \Nl}$, and $\{J e_n\ot e_m\}_{n,m\in \Nl}$ are orthonormal bases of $\Hil\ot\Hil$. One can now simply compute the Hilbert--Schmidt norm of $\Cross_J(T)$ with respect to one of these bases, namely,
	\begin{align*}
		\|\Cross_J(T)\|_2^2
		&=
		\sum_{n,m,k,l=1}^{\infty}|\ip{e_k\ot e_l}{\Cross_J(T)(e_n\ot e_m)}|^2
		\\
		&=\sum_{n,m,k,l=1}^{\infty}|\ip{e_l\ot Je_m}{T(Je_k\ot e_n)}|^2
		=
		\|T\|_2^2.
	\end{align*}
	Hence, $\Cross_J$ is an isometry of $\CI_2$. Since $\Cross_J$ is bijective on $\CrossSetb$, also its restriction to $\CI_2$ is bijective and hence $\Cross_J$ is a unitary on $\CI_2$. The claim about ${\Cross_S}|_{\CI_2}$ follows.
	
	$b)$ We first compute the adjoint of the  $S$-crossing map. Since $(\Deltatout_{1/2})|_{\CI_2}$ is self-adjoint, and $(\Cross_J)|_{\CI_2}$ is unitary, we have (omitting the restriction symbols)
	\begin{align*}
		\Cross_S^*
		&=
		(\Deltatout_{1/2}\Cross_J)^*
		=
		\Cross_J^{-1}\Deltatout_{1/2}
		=
		\St\Cross_J\St\Deltatout_{1/2}
		\\
		&=
		\St\Deltatout_{-1}\Deltatout_{1/2}\Cross_J\St
		=
		\Deltatin_{1}\St\Cross_S\St
		=
		\Deltatin_{1}\Cross_S^{-1},
	\end{align*}
	and hence
	\begin{align*}
		\Sh^*\Sh
		&=
		(\St\Cross_S)^*\St\Cross_S
		=
		\Cross_S^*\Cross_S
		=
		\Deltatin_{1}
		=
		\Deltah,
	\end{align*}
	since $\St^*=\St$ as a consequence of $\Tr$ being hermitian. The other claims follow from standard facts about modular theory of standard subspaces.
\end{proof}
We have thus constructed a map
\begin{align}
	H_S \longmapsto \Hh_S\cap\CI_2 = H_{\Sh|_{\CI_2}}
\end{align}
from standard subspaces of $\Hil$ (for bounded $S\in\CS(\Hil)$) to standard subspaces of the Hilbert--Schmidt operators $\CI_2$, which allows us to translate questions on standard subspaces into questions on crossing symmetry, and vice versa.

\begin{remark}
	We conclude this section by pointing out that the crossing map $\Cross_J$ for antiunitary $J$ has various formal similarities to the classical Fourier transform $\CF$, in particular on the Hilbert--Schmidt operators~$\CI_2$: We have just seen that the restriction of $\Cross_J$ to $\CI_2$ is unitary (w.r.t.\ the Hilbert--Schmidt scalar product), of order four (Prop.~\ref{prop:bij}~c)), and the inversion formula $\Cross_S^{-1}(T)^* = \Cross_S(T^*)$ from Prop.~\ref{prop:bij}~d) (which even holds for general bounded $S$), is analogous to the Fourier inversion formula $\CF(\bar f)=\overline{\CF^{-1}(f)}$.
	In Section \ref{sec:catcross}, we shall further comment on the relation between $\Cross_S$ (for $\dim\Hil < \infty$) and the subfactor theoretical Fourier transform, see Theorem \ref{thm:HilbfCrossing} and Remark \ref{rmk:CatCrossisF}.
\end{remark}

\section{Crossing symmetry and endomorphisms of standard subspaces}\label{sec:endos}

We now describe a relation between the $S$-crossing map $\Cross_S$ and $S$-crossing symmetry with the standard subspace $H_S=\ker(S-1)$, see Remark~\ref{remark:modulartheory} for our notation. Given the bijection $S_H\leftrightarrow H_S$ described there, it can be expected that an $S$-crossing symmetric twist has nice properties with respect to the standard subspace $H=H_S$. As we shall see, it defines endomorphisms of $H$. Conversely, endomorphisms of $H$ can be used to generate crossing symmetric operators.

As in the previous section, we use the notation $\Hh_S$ or just
\begin{align}
	\Hh := \{T\in\CrossSet_{S_H}(\Hil\tp{2})\,:\,\Cross_{S_H}(T)=T^*\}
\end{align}
for the real space of all ($S_H$-crossable and) $S_H$-crossing symmetric operators, where $H\subset\Hil$ is any standard subspace.

In the following, we want to relate $\Hh$ to endomorphisms of $H$, i.e.\ elements of
\begin{align}
	\CE(H)
	:=
	\{V\in\CB(\Hil)\,:\,VH\subset H\}.
\end{align}
The elements of $\CE(H)$ can be characterized in different ways \cite{ArakiZsido:2005}. For example, a bounded operator $V$ lies in $\CE(H)$ if and only if it commutes with~$S_H$ in the sense $J_HVJ_H\Delta_H^{1/2}\subset\Delta_H^{1/2}V$.

To connect $\Hh$ and $\CE(H)$, it will be useful to work with the left/right creation and annihilation operators, $\xi\in\Hil$,
\begin{align*}
	a^*_{L/R}(\xi):\Hil\to\Hil\ot\Hil,\qquad a^*_L(\xi)\psi:=\xi\ot\psi,\quad a^*_R(\xi)\psi:=\psi\ot\xi,
\end{align*}
and their adjoints $a_{L}(\xi)(\psi_1\ot\psi_2)=\langle\xi,\psi_1\rangle\psi_2$ and  $a_{R}(\xi)(\psi_1\ot\psi_2)=\langle\xi,\psi_2\rangle\psi_1$.

\begin{theorem}\label{thm:CrossingEnd}
	Let $T\in \CB(\Hil\tp{2})$ and define $V_{\psi_1,\psi_2}(T) \in \CB(\Hil)$ by
	\begin{align}
		V_{\psi_1,\psi_2}(T) :=\frac{1}{2}\left( a_L(\psi_1)Ta_R^*(\psi_2)+a_L(\psi_2)Ta_R^*(\psi_1)\right), \quad \psi_1,\psi_2\in \Hil.
	\end{align}
	Then $T\in\Hh$ if and only if $T\in\CrossSet_{S_H}(\Hil\tp{2})$ and $V_{\psi_1,\psi_2}(T) \in \CE(H)$ for every $\psi_1,\psi_2\in \Hil$. 
	
	Furthermore, there exists a bijection between $\Hh$ and the bounded $\Rl$-bilinear endomorphism-valued maps $V:\Hil\times \Hil \to\CE(H)$ satisfying
	\begin{enumerate}
		\item $V(i\psi_1,\psi_2)=-V(\psi_1,i\psi_2)$;
		\item the sesquilinear form $\sigma:(\Hil \ot \Hil)\times (\Hil \ot \Hil) \to \Cl$ defined by
		\begin{align}
			\sigma(\psi_1\ot \xi,\eta\ot \psi_2)=\ip{\xi}{\left( V(\psi_1,\psi_2)+i V(i\psi_1,\psi_2)\right)\eta},
		\end{align}
		which is well defined thanks to a), is bounded.
	\end{enumerate}
\end{theorem}
\begin{proof}
	
	$(\Rightarrow)$ Let $T\in\Hh\subset\CrossSet_{S_H}(\Hil\tp{2})$.  Notice now that for any $\eta \in \dom(S_H)$ and $\xi \in \dom(S_H^*)$,
	\begin{align*}
		\ip{\xi}{a_L(\psi_1)Ta_R^*(\psi_2) S_H\eta}&=\ip{\psi_1\ot \xi}{T (S_H\eta\ot \psi_2)}\\
		&=\ip{\eta\ot \psi_1}{\Cross_{S_H}(T) (\psi_2\ot S_H^\ast \xi)}\\
		&=\ip{\eta\ot \psi_1}{T^* (\psi_2\ot S_H^\ast\xi)}\\
		&=\ip{a_L(\psi_2)Ta_R^*(\psi_1)\eta}{S_H^*\xi}.
	\end{align*}
	It follows that $\xi\mapsto \ip{a_L(\psi_2)Ta_R^*(\psi_1)\eta}{S_H^*\xi}$ is a continuous functional. Hence, $a_L(\psi_2)Ta_R^*(\psi_1)\eta\in \dom(S_H)$ and $S_Ha_L(\psi_2)Ta_R^*(\psi_1) \eta=a_L(\psi_1)Ta_R^*(\psi_2) S_H\eta$ for all $\eta \in \dom(S_H)$. In other words, $a_L(\psi_1)Ta_R^*(\psi_2)S_H \subset S_Ha_L(\psi_2)Ta_R^*(\psi_1)$ and it follows that $V_{\psi_1,\psi_2}(T)S_H\subset S_H V_{\psi_2,\psi_1}(T)$, thus the operator $V_{\psi_1,\psi_2}(T)$ that is symmetrized in $\psi_1,\psi_2$ is an endomorphism of $H$.
	
	$(\Leftarrow)$ Set $V_\psi(T):=V_{\psi,\psi}(T)$, $\psi\in\Hil$. For any $\eta \in \dom(S_H)$ and $\xi\in\dom(S_H^*)$, it follows that $V_\psi(T)\eta\in \dom(S_H)$ and
	\begin{align*}
		\ip{\eta\ot \psi}{T^*(\psi\ot \xi)}&=\overline{\ip{\psi\ot \xi}{T(\eta\ot \psi)}}\\
		&=\overline{\ip{\xi}{V_\psi(T)\eta}}\\
		&=\ip{S_H^*\xi}{V_\psi(T)S_H \eta}\\
		&=\ip{\psi\ot S_H^*\xi}{T(S_H\eta\ot \psi)}\\
		&=\ip{\eta\ot \psi}{\Cross_{S_H}(T)(\psi\ot \xi)}.
	\end{align*}
	Using polarization, we conclude that $T^* = \Cross_{S_H}(T)$.
	
	For the additional information, it is clear that, given $T\in \Hh$, it defines a map $V:\Hil\times\Hil\to \CE(H)$ with the required properties by setting $V(\psi_1,\psi_2)=V_{\psi_1,\psi_2}(T)$. For the converse, let $V:\Hil\times\Hil\to \CE(H)$ be as in the statement. As already mentioned in b), one can check that property a) guarantees that $\sigma:(\Hil\odot \Hil )\times (\Hil\odot \Hil) \to\bC$ defined in b) is a well-defined sesquilinear form and property b) demands that $\sigma$ is bounded. Hence, $\sigma$ can be continuously and sesquilinearly extended to $(\Hil\ot \Hil) \times (\Hil\ot \Hil)$ and, by virtue of the Riesz Representation Theorem, $\sigma$ defines a unique $T \in \CB(\Hil\tp{2})$ such that $\sigma(\Psi_1,\Psi_2)=\ip{\Psi_1}{T\Psi_2}$ for all $\Psi_1,\Psi_2\in \Hil\ot \Hil$. It is not hard to see that from the defining property of $T$ it follows that $V_{\psi_1,\psi_2}(T)=V(\psi_1,\psi_2)\in \CE(H)$.
	From $(\Leftarrow)$ above, it follows that $T\in \Hh$.
\end{proof}

The polarization argument in Theorem \ref{thm:CrossingEnd} makes it clear that one can recover the operator $T$ from the family of operators $V_\psi(T):=V_{\psi,\psi}(T)$, $\psi\in \Hil$, and the following analogous version holds.

\begin{corollary}\label{cor:familyEnd}
	Let $T\in \CB(\Hil\tp{2})$. Then $T\in\Hh$ if and only if $T\in\CrossSet_{S_H}(\Hil\tp{2})$ and $V_\psi(T) = a_L(\psi)Ta_R^*(\psi)\in \CE(H)$ for every $\psi\in \Hil$. Furthermore, $V_\psi(T)=V_\psi(\tilde T)$ for all $\psi\in \Hil$ if and only if $T=\tilde T$.
	
	In particular, if $T\in\Hh$ satisfies $V_\psi(T)=\la\|\psi\|^2\cdot 1$ for all $\psi\in\Hil$ and some $\la\in\Cl$, then $T=\la F$.
\end{corollary}

The last statement in this corollary follows from the following example, which also shows how one can use Theorem \ref{thm:CrossingEnd} to produce examples of crossing symmetric operators from families of endomorphisms.

\begin{example}\label{ex:crosssymfromend1}
	Let $R\in \CB(\Hil)$ be self-adjoint, $H\subset\Hil$ a standard subspace, and $V\in \CE(H)$. Then $(R\ot V)F\in \CrossSet_{S_H}(\Hil\tp{2})$ is the unique $S_H$-crossing symmetric operator such that $V_\psi\left((R\ot V)F\right) = \langle\psi,R\psi\rangle\cdot V$ for all $\psi\in \Hil$. In fact, a simple computation shows that
	\begin{align*}
		\ip{\eta}{\ip{\psi}{R\psi}V\xi}&=\ip{\psi}{R\psi}\ip{\eta}{V\xi}\\
		&=\ip{\psi\ot\eta}{(R\ot V)F(\xi\ot\psi)}\\
		&=\ip{\eta}{V_\psi\left((R\ot V)F\right) \xi}.
	\end{align*}
	Since $\{\ip{\psi}{R\psi}V\}_{\psi\in \Hil}\subset \CE(H)$, it follows from Corollary \ref{cor:familyEnd} that $(R\ot V)F$ is the unique $S_H$-crossing symmetric operator such that $V_\psi\left((R\ot V)F\right) = \langle\psi,R\psi\rangle\cdot V$ for all $\psi\in \Hil$.
\end{example}

In situations typical for quantum field theory, one considers standard subspaces with additional structure such as symmetry groups, cf.\ Section \ref{sec:symmetry}. A prominent example of such a scenario is a {\em standard pair}, consisting of a standard subspace $H\subset\Hil$ and a unitary one-parameter group $U(x)=e^{iPx}$, $x\in\Rl$, such that
\begin{align}
	U(x)H\subset H,\quad x\geq0,\qquad P>0.
\end{align}
By a well-known theorem of Borchers \cite{Borchers:1992}, the two one-parameter groups $U(x)$ and $\Delta_H^{it}$ generate a representation of the affine group of $\Rl$. Thanks to the underlying canonical commutation relations, there exists a unique irreducible standard pair (in the sense that the representation of the affine group is irreducible). It can be described on $\Hil=L^2(\Rl,d\te)$ as
\begin{align*}
	(U(x)\psi)(\te)=e^{ixe^\te}\psi(\te),\quad
	(\Delta_H^{it}\psi)(\te)=\psi(\te-2\pi t),\qquad
	(J_H\psi)(\te)=\overline{\psi(\te)}.
\end{align*}
Concretely, the standard subspace $H=\ker(J_H\Delta_H^{1/2}-1)$ is a the real subspace of the Hardy space $\Hardy^2(\Strip_\pi)$ on the strip $\Strip_\pi=\{\zeta\in\Cl\,:\,0<\Im\zeta<\pi\}$ fixed by the boundary condition $\overline{\psi(\te+i\pi)}=\psi(\te)$, $\te\in\Rl$ \cite[Lemma~A.1]{LechnerLongo:2014}.

In this case, the endomorphisms of $H$ that also commute with the one-parameter group $U(x)$ (denoted $\CE(U,H)\subset\CE(H)$) are explicit, and known as \emph{Longo--Witten endomorphisms}, see Theorem~2.3 in \cite{LongoWitten:2011}. Namely, $\CE(U,H)$ consists of the multiplication operators
\begin{align}\label{eq:LWEndo}
	(\varphi(P)\psi)(\te) = \varphi(e^\te)\cdot\psi(\te),
\end{align}
where $\varphi\in \Hardy^\infty(\Strip_\pi)$ is the (boundary value of) a bounded analytic function in $\Strip_\pi$, satisfying $\varphi(\te+i\pi)=\overline{\varphi(\te)}$, $\te\in\Rl$.

We now show which crossing symmetric operators are generated by these Longo--Witten endomorphisms.

\begin{example}\label{ex:Longo-Witten1}
	Consider the irreducible standard pair $(U,H)$ on $\Hil=L^2(\Rl,d\te)$ defined above, and let $\varphi\in \Hardy^\infty(\Strip_\pi)$ satisfy $\varphi(\theta+i\pi)=\overline{\varphi(\theta)}$. Consider the real bilinear map $V_\varphi:L^2(\Rl)\times L^2(\Rl) \to \B(L^2(\Rl))$ given by the multiplication operators
	\begin{equation}
		(V_\varphi(\psi_1,\psi_2)\xi)(\theta)=
		\left(\int \Re\left(\overline{\psi_1(\tilde{\theta})}\psi_2(\tilde{\theta}) \right)\varphi(\theta-\tilde{\theta})d\tilde{\theta}\right)\cdot\xi(\theta).
	\end{equation}
	Note that the conditions on $\varphi$ imply that $V_\varphi(\psi_1,\psi_2)$ is a bounded operator satisfying $V_\varphi(\psi_1,\psi_2)H\subset H$ and commuting with $U$, i.e.\ $V_\varphi(\psi_1,\psi_2)\in \CE(U,H)$. It is not hard to see that
	\begin{align}\label{goodtwist}
		(T_\varphi \Psi)(\theta_1,\theta_2) := \varphi(\theta_2-\theta_1)\Psi(\theta_2,\theta_1)
	\end{align}
	satisfies $V_\varphi(\psi_1,\psi_2)=V_{\psi_1,\psi_2}(T_\varphi)$ and therefore $T_\varphi\in\Hh_S$ is $S_H$-crossing symmetric.
\end{example}

\begin{remark}\label{remark:ModvsQFT-Crossing}
	The crossing symmetric operators $T_\varphi$ defined in \eqref{goodtwist} occur frequently in quantum field theory, where they give the simplest examples of elastic two-body scattering matrices for integrable models on two-dimensional Minkowski space \cite{Lechner:2003}. In particular, for these operators our notion of crossing symmetry coincides with the crossing symmetry from scattering theory \cite{Martin:1969_2}. For previous investigations regarding the relation of Longo--Witten endomorphisms and two-body S-matrices, see \cite{LechnerSchlemmerTanimoto:2013}.
\end{remark}

Since crossing symmetry is an $\Rl$-linear condition we can straightforwardly generalize Example \ref{ex:crosssymfromend1} to $T:=\sum_{n=1}^N (R_n\ot V_n)F$, where $R_n=R_n^*\in \B(\Hil)$ and $V_n\in\CE(H)$. It is however more interesting to go beyond the limitations of this construction by means of operator-valued integrals, see also \cite{Tanimoto:2011-1} for closely related techniques.

Let $E$ be a projection-valued measure on the Borel $\sigma$-algebra of $\Rl$ and let $A:\supp(E)\to\B(\Hil)$ be weakly integrable with respect to the measure $\mu_\psi=\ip{\psi}{E\psi}$ for all $\psi\in \Hil$. Then, by polarization,
$$\sigma(\psi_1\ot \psi_2,\psi_3\ot \psi_4):=\int \ip{\psi_1}{dE_\lambda\psi_3}\ip{\psi_2}{A(\lambda)\psi_4}$$
is well defined in pure tensors. In case $\sigma$ extends to a bounded sesquilinear form on $\Hil\ot\Hil$ we say that $A$ is integrable with respect to the projection-valued measure $B\mapsto E(B)\ot 1$. In this case, $\sigma$ defines through the formula $\sigma(\cdot,\cdot)=\ip{\cdot}{T\cdot}$ a unique bounded linear operator which we denote by
\begin{equation}
	T=\int dE_\lambda\ot A(\lambda) \in \B(\Hil\ot \Hil).
\end{equation}

One immediate consequence of this definition is that, in pure tensor products,
\begin{align}
	\ip{\psi_1\ot \psi_2}{T\psi_3\ot\psi_4}&=\int \ip{\psi_1}{A(\lambda)\psi_3}\ip{\psi_2}{dE_\lambda\psi_4}\\
	&=\overline{\int \ip{\psi_3}{A(\lambda)^\ast\psi_1}\ip{\psi_4}{dE_\lambda\psi_2}}\\
	&=\overline{\ip{\psi_3\ot\psi_4}{\int dE_\lambda\ot A(\lambda)^\ast \psi_1\ot \psi_2}},
\end{align}
which yields $T^\ast=\int dE_\lambda\ot A(\lambda)^\ast$.

We are now in the position to state the following proposition about producing crossing symmetric operators by integrating one-parameter families of endomorphisms:

\begin{proposition}\label{prop:tensorintergral}
	Let $H\subset\Hil$ be a standard subspace, $E$ a projection-valued measure on the Borel $\sigma$-algebra of $\Rl$, and $V:\supp(E)\to \CE(H)$ integrable w.r.t.\ $E\ot1$. Then
	\begin{equation}\label{integraltwist}
		T:=\left(\int dE_\lambda\ot V(\lambda) \right)\cdot F \in \B(\Hil\ot \Hil)
	\end{equation}
	is $S_H$-crossing symmetric.
\end{proposition}
\begin{proof}
	By assumption, $V(\lambda)\in\CE(H)$ for all $\la$, and hence $S_HV(\lambda)S_H\subset V(\lambda)$. Hence, for any $\psi_2,\psi_3\in \Hil$, $\psi_1\in \dom(S_H)$, and $\psi_4\in \dom(S_H^*)$, it holds
	\begin{align}
		Q_{S_H}(T)(\psi_1\ot \psi_2,\psi_3\ot\psi_4)
		&=
		\ip{\psi_2 \ot S_H^*\psi_4 }{T (S_H\psi_1\ot \psi_3)}\\
		&=
		\int \ip{S_H^*\psi_4}{V(\lambda)S_H\psi_1}\ip{\psi_2}{dE_\lambda\psi_3}\\
		&=
		\int\ip{V(\lambda)\psi_1}{\psi_4}\ip{\psi_2}{dE_\lambda\psi_3}\\
		&=\ip{\psi_1\ot \psi_2}{F\int dE_\lambda\ot V(\lambda)^* \ \psi_3\ot\psi_4}
		\\
		&=\ip{\psi_1\ot \psi_2}{T^*(\psi_3\ot\psi_4)},
	\end{align}
	showing that $T$ is $S_H$-crossable and $\Cross_{S_H}(T)\subset T^*$. Hence $\Cross_{S_H}(T) = T^*$.
\end{proof}

We now show how the twists \eqref{goodtwist} can be obtained from such an integral based on a natural measure and a natural family of endomorphisms.

\begin{example}
	Consider the irreducible standard pair $(U,H)$ on $\Hil=L^2(\Rl,d\theta)$ as in Example \ref{ex:Longo-Witten1}, a function $\varphi\in \Hardy^\infty(\Strip_\pi)$ with $\varphi(\theta+i\pi)=\overline{\varphi(\theta)}$, the corresponding Longo--Witten endomorphism $V := \varphi(P)$ from \eqref{eq:LWEndo}, and the self-adjoint operator $(\Theta\psi)(\te) := \te\psi(\te)$ on $L^2(\Rl,d\te)$.
	
	Then the $S_H$-crossing symmetric twist $T$ defined by \eqref{integraltwist} with $E:=E^\Theta$ the spectral measure of $\Theta$ and $V(\lambda):=\Delta^{i\frac{\lambda}{2\pi}}V\Delta^{-i\frac{\lambda}{2\pi}}$, coincides with \eqref{goodtwist}, i.e.
	\begin{align*}
		\left(\int dE^{\Theta}_\lambda\ot V(\lambda) \ F\right)\Psi(\te_1,\te_2)
		=
		\varphi(\te_2-\te_1)\cdot\Psi(\te_2,\te_1),\qquad \Psi\in\Hil^{\ot2}.
	\end{align*}
	To prove this, we consider $\psi_1,\ldots,\psi_4\in\Hil$ and calculate
	\begin{align}
		\ip{\psi_1\ot\psi_2}{T(\psi_3\ot\psi_4)}
		&=
		\int \overline{\psi_2(\theta_2)} \varphi(\theta_2-\lambda)\psi_3(\theta_2) \ip{\psi_1}{dE^{\Theta}_\lambda \psi_4} d\theta_2\\
		&= \int \overline{\psi_2(\theta_2)} \varphi(\theta_2-\theta_1)\psi_3(\theta_2) \overline{\psi_1(\theta_1)} \psi_4(\theta_1)d\theta_1 d\theta_2\\
		&= \int \overline{\psi_1(\theta_1)}\overline{\psi_2(\theta_2)} \varphi(\theta_2-\theta_1)\psi_3(\theta_2)  \psi_4(\theta_1)d\theta_1 d\theta_2\,
	\end{align}
	which implies the claim.
\end{example}

\section{Crossing symmetry and symmetry groups}\label{sec:symmetry}

As explained in the introduction, in applications of crossing symmetry to quantum field theory, one is often interested in crossing symmetric operators (w.r.t.\ some Tomita operator $S\in\CS(\Hil)$) that have additional properties, like solving a Yang--Baxter equation or exhibiting further symmetries. In this section, we focus on the latter point. Motivated by the Bisognano-Wichmann property of QFT \cite{BisognanoWichmann:1976} and its generalizations \cite{Borchers:1992,BrunettiGuidoLongo:2002,MorinelliNeeb:2023}, we are mostly interested in situations in which the representation under consideration contains the modular group of the standard subspace defining the crossing map.

\begin{definition}\label{def:G-invariant}
	Let $S\in\CS(\Hil)$. A crossing symmetric operator $T=\Cross_S(T)^*\in\Hh_S$ is called {\em invariant} under a unitary representation $U$ of some group $G$ on $\Hil$, or simply {\em $U$-invariant}, when
	\begin{align}
		[T,U_2(g)]=0, \quad g\in G,\qquad U_2(g):=U(g)\ot U(g).
	\end{align}
	The set of all $U$-invariant $S$-crossing symmetric operators is denoted~$\Hh_S^U$.
\end{definition}

Symmetries of $T$ transfer to $\Cross_S(T)$ in an $S$-twisted manner.

\begin{lemma}\label{TcommutingUxU}
	Let $S\in\CS(\Hil)$, a crossable $T\in\CrossSet_S$, and let $U\in\CU(\Hil)$ such that $[T,U\ot U]=0$. Then
	\begin{enumerate}
		\item $(S^*US^*\ot U)\Cross_S(T)(U^*\ot S^*U^*S^*) \subset \Cross_S(T)$,
		
		\item If $\Cross_S(T)=T^*$, then
		\begin{align*}
			T\subset
			(1\ot SUSU^*)T(USU^*S\ot 1).
		\end{align*}
		
		\item If $\Cross_S(T)=T^*$ and $UH\subset H'$, then
		\begin{align*}
			T\subset
			(1\ot \Delta^{-1})T(\Delta\ot 1).
		\end{align*}
		
		\item If $\Cross_S(T)=T$ and $[T,\Delta^{it}\ot\Delta^{it}]=0$, then
		\begin{align}
			F(J\ot J)T(J\ot J)F=T.
		\end{align}
	\end{enumerate}
\end{lemma}
\begin{proof}
	a) follows by direct calculation: For vectors $\varphi_1,\varphi_2,\psi_1,\psi_2$ with $\varphi_1\in\dom(SU^*S)$ and $\psi_2\in\dom(S^*U^*S^*)$, we get
	\begin{align*}
		\langle SU^*S\varphi_1\ot U^*\varphi_2,\Cross_S(T)(U^*\psi_1\ot S^*U^*S^*\psi_2)\rangle
		&=
		\langle U^*\varphi_2\ot U^*S^*\psi_2,T(U^*S\varphi_1\ot U^*\psi_1)\rangle
		\\
		&=
		\langle \varphi_2\ot S^*\psi_2,T(S\varphi_1\ot \psi_1)\rangle
		\\
		&=
		\langle\varphi_1\ot\varphi_2,\Cross_S(T)(\psi_1\ot\psi_2)\rangle,
	\end{align*}
	which implies the claim.
	
	b) follows from a) by inserting $\Cross_S(T)=T^*$ and simplifying with the help of $[T,U\ot U]=0$.
	
	c) considering the Tomita operators of $UH$ and $H'$, the assumption on $U$ implies $USU^*\subset S^*$. Hence the statement follows from b).
	
	d) The argument is an adaptation of \cite[Lemma 3.4]{CorreaDaSilvaLechner:2023}. From the condition $[T,\Delta^{it}\ot\Delta^{it}]=0$, it follows that $[\Cross_S(T),\Delta^{it}\ot\Delta^{it}]=0$. Using this latter fact, Proposition \ref{prop:KMS-likeCrossing} states that the function $t\mapsto T^{\psi_1,\psi_2}_{\psi_3,\psi_4}(t)$ has an analytic extension to $\Strip_{1/2}$ and bounded on $\overline{\Strip_{1/2}}$ satisfying $T^{\psi_1,\psi_2}_{\psi_3,\psi_4}(t+\frac{i}{2})=\Cross_S(T)^{J\psi_3,\psi_1}_{\psi_4,J\psi_2}(-t)$ for any vectors $\psi_i\in \Hil$, $1\leq i\leq 4$. Then, since $\Cross_S(T)=T$
	\begin{align}
		T^{\psi_1,\psi_2}_{\psi_3,\psi_4}(t)&=T^{\psi_2,J\psi_4}_{J\psi_1,\psi_3}\left(-t+\tfrac{i}{2}\right)\\
		&=T^{J\psi_4,J\psi_3}_{J\psi_2,J\psi_1}\left(-\left(-t+\tfrac{i}{2}\right)+\tfrac{i}{2}\right)\\
		&=\left(F (J\ot J) T(J\ot J)F\right)^{\psi_1,\psi_2}_{\psi_3,\psi_4}(t).
	\end{align}
	We conclude, at $t=0$, that $F(J\ot J)T(J\ot J)F$.
\end{proof}

We next consider the interplay of symmetries of $T\in\Hh_S$ and the associated endomorphisms $V_\psi(T)$ of $H$ defined in Corollary \ref{cor:familyEnd}.

\begin{lemma}\label{lemma:TUprops}
	Let $S\in\CS(\Hil)$, $T\in\Hh_S$, $U\in\CU(\Hil)$, and $\psi\in\Hil$.
	\begin{enumerate}
		\item If $[T,U\ot U]=0$ and $\psi$ is an eigenvector of $U$, then $[V_\psi(T),U]=0$.
		\item If $(1\ot B)T(C\ot 1)\subset T$ for closable operators $B,C$ on $\Hil$, then one has $B V_\psi(T) C\subset V_\psi(T)$.
		\item If $[T,U\ot U]=0$, then $U^*V_\psi(T) U$ is also an endomorphism of $H$.
	\end{enumerate}
\end{lemma}
\begin{proof}
	a) This is a direct consequence of
	\begin{align}
		V_\psi(T)
		&=
		a_L(\psi)(U^*\ot U^*)T(U \ot U)a_R^*(\psi)
		=
		U^*V_{U\psi}(T)U
		=
		U^*V_{\psi}(T)U.
	\end{align}
	An analogous calculation shows b).
	
	c) By the preceding lemma, we know that the situation in b) holds with $SUSU^*$ instead of $B$ and its adjoint $USU^*S$ instead of $C$. Hence, using that $SV_\psi(T)S\subset V_\psi(T)$ we conclude that
	\begin{align*}
		SUSU^* V_\psi(T) USU^* S \subset V\Rightarrow SU^* V_\psi(T) US \subset U^*SV_\psi(T)SU\subset U^*V_\psi(T)U
	\end{align*}
	and the claim follows.
\end{proof}

Part b) of Lemma \ref{TcommutingUxU} states that for any unitary $U$ such that $T\in \Hh_S$ commutes with $U\ot U$ one also has $T=(1\ot SUSU^*)T(USU^*S\ot 1)$. In case~$U$ is an endomorphism of $H$, i.e.\ $SUS\subset U$, this constraint is automatically satisfied (in particular for $U=\Delta^{it}$). If, on the other hand, the unitary $U$ maps $H$ into its symplectic complement $H'$, then $T = (1\ot \Delta^{-1})T(\Delta\ot 1)$ by part c), which forces the maps $t\mapsto T^{\psi_1,\psi_2}_{\psi_3,\psi_4}(t)$ to be constant and $V_\psi(T)$ to be {\em auto}morphisms of $H$.

Instead of a general analysis of invariant crossing symmetric operators, in this article we restrict ourselves to the discussion of a few instructive examples, delegating a more systematic investigation to future work. The first example shows that the whole unitary group is too large to allow for interesting crossing symmetric operators.

\begin{example}
	Consider a Hilbert space $\Hil$, its unitary group $G={\CU}(\Hil)$ in its defining representation $U$, and $S\in\CS(\Hil)$ arbitrary. Then $\Hh_S^U=\Rl F$, the real multiples of the flip.
\end{example}
One can prove this by direct calculation starting with Schur-Weyl duality to determine the commutant of $U_2$, but we prefer to give a shorter indirect argument emphasizing the connection to endomorphisms of standard subspaces.
\begin{proof}
	Let $T\in\Hh_S^U$ and $\psi\in\Hil$ be an arbitrary unit vector. There exists a unitary $U\in\CU(\Hil)$ such that $UH=H'$. By Lemma~\ref{lemma:TUprops}~b), we know that $V_\psi(T)$ commutes with $\Delta_H$. As it is an endomorphism of $H$, it follows that $V_\psi(T)$ also commutes with $J_H$. Moreover, by Lemma~\ref{lemma:TUprops}~c) for any $U\in\CU(\Hil)$, we know that $V_\psi(T)$ is an endo- and hence automorphism of $UH$ by the same argument. So $V_\psi(T)$ commutes with $J_{UH}=UJ_HU^*$ for all $U\in\CU(\Hil)$. Hence $V_\psi(T)$ commutes with the group $\CJ$ generated by all antiunitary involutions on~$\Hil$, which implies $V_\psi(T)=\la1$.
	
	We claim $V_\xi(T)=\la\|\xi\|^21$ for all $\xi\in\Hil$. Indeed, if $\|\xi\|=1$, there exists a unitary $W$ such that $W\xi=\psi$. Then, since $T=(W^*\ot W^*)T(W\ot W)$, the argument of Lemma~\ref{lemma:TUprops}~a) shows $V_\xi(T)=W^*V_\psi(T)W=\la1$. For general $\xi$ we then have $V_\xi(T)=\la\|\xi\|^21$, as claimed. Now Corollary~\ref{cor:familyEnd} yields the result.
\end{proof}

In the second example, we pick an antiunitary involution $J$ on $\Hil$ and consider the ``$J$-orthogonal'' group
\begin{align}
	G_J:=\{U\in\CU(\Hil)\,:\,[U,J]=0\},
\end{align}
which acts irreducibly on $\Hil$. This example is most interesting in case $1<N:=\dim\Hil<\infty$ is finite, i.e.\ when $G\cong O(N)$ is the orthogonal group. In this case, we may consider an involution $S=J\Delta^{1/2}\in\CS(\Hil)$, the vector $\xi_S$ \eqref{def:xiS}, and the associated rescaled projection $P_S=\Cross_S(1)$ \eqref{eq:PS}.

\begin{example}
	Consider $G_J\cong O(N)$ on $\Hil\cong\Cl^N$ in its defining representation~$U$. Let $S=J\Delta^{1/2}\in\CS(\Hil)$ be an antilinear involution with antiunitary part $J$. Then
	\begin{align}
		\Hh_S^U
		&=
		\begin{cases}
			\{z\cdot1+\overline z \,P_J+x\,F\,:\,z\in\Cl,x\in\Rl\}  & \Delta=1
			\\
			\{xF+iy(1-P_J)\,:\,x,y\in\Rl\} & \Delta\neq1,\quad N=2,\\
			\Rl F & \Delta\neq1,\quad N>2.
		\end{cases}
	\end{align}
\end{example}
\begin{proof}
	By classical invariant theory we know that the commutant of the diagonal representation $U_2$ is spanned by $1,F,P_J$ \cite{GoodmanWallach:2009}. Hence we may represent any $T\in\Hh_S^U$ in the form $T=z1+wP_J+xF$, $z,w,x\in\Cl$.
	
	For $\Delta=1$, we find with the help of $\Cross_J(1)=P_J$ and $\Cross_J(P_J)=1$ \eqref{eq:CrossId}
	\begin{align*}
		\Cross_J(T)-T^*
		&=
		(z-\overline w)P_J+(w-\overline z)+(x-\overline x)F,
	\end{align*}
	which vanishes if and only if $w=\overline z$ and $x\in\Rl$, as claimed.
	
	For general modular operator $\Delta\neq1$, the condition $T\in\Hh_S^U$ reads $X=0$, where 
	\begin{align}
		X&:=
		zP_S
		-\overline w\,P_J
		+w\,\Delta^{1/2}\ot\Delta^{-1/2}-\overline{z} \,1+(x-\bar x)F.
	\end{align}
	Suppose there exists $1\neq\la\in\sigma(\Delta)$ with multiplicity bigger than one, and pick two orthogonal eigenvectors $v_\la,\tilde v_\la$ of unit length. Then 
	\begin{align}
		\langle v_\la\ot J\tilde v_\la,X(v_\la\ot J\tilde v_\la)\rangle
		&=
		w\la-\overline z,
	\end{align}
	because $J\tilde v_\la \perp v_\la$. Since also $\la^{-1}$ lies in the spectrum of $\Delta$, this implies $w=z=0$, and hence $x\in\Rl$, as claimed. In case the eigenvalues of $\Delta$ different from~$1$ have multiplicity 1, but there exist $\la,\mu\in\sigma(\Delta)$ with $\mu\not\in\{\la,\la^{-1}\}$, we pick corresponding normalized eigenvectors $v_\la$, $v_\mu$ and arrive at
	\begin{align}
		\langle v_\la\ot v_\mu,X(v_\la\ot v_\mu)\rangle
		&=w\frac{\la^{1/2}}{\mu^{1/2}}-\overline z,
	\end{align}
	which again implies $z=w=0$ and $x\in\Rl$ as before.
	
	This shows that $T$ is a real multiple of the flip whenever $\Delta\neq1$ and $N>2$. The remaining case is $N=2$ and $\sigma(\Delta)=\{\la,\la^{-1}\}$ with $\la\neq1$. In this case computations in the tensor basis corresponding to the eigenbasis of $\Delta$ show that~$X$ vanishes if and only if $x\in\Rl$ and $z=-w=:iy$ is purely imaginary, $y\in\Rl$. This proves the claim.
\end{proof}

In the previous example, we obtained besides real multiples of the flip a non-trivial crossing symmetric solution $T = i (P_J-1)$ for $N=2$ and modular operator $\Delta=\begin{pmatrix}
	\la&0\\0&\la^{-1}
\end{pmatrix}$ for $\la\neq1$. In canonical identification $M_2\ot M_2\cong M_4$, it takes the form
\begin{align}\label{eq:klR}
	i\begin{pmatrix}
		-1\\
		&&1\\
		&1\\
		&&&-1
	\end{pmatrix}.		
\end{align}

Much richer examples occur for infinite-dimensional representations. We therefore move on to non-compact groups and discuss an example that is of prominent interest in physics. Namely, take $G=SO(1,1)\rtimes\Rl^2$ to be the Poincar\'e group in dimension two, which can be viewed as the group of maps $(x,\la)$ (with $x\in\Rl^2$, $\la\in\Rl$)
\begin{align}
	(x,\la):\Rl^2\to\Rl^2,\qquad y\mapsto \La_\la y+x, \qquad
	\La_\la
	=
	\begin{pmatrix}
		\cosh\la&\sinh\la\\
		\sinh\la&\cosh\la
	\end{pmatrix}.
\end{align}
As representation, we consider the irreducible positive energy representation $U_m$ of mass $m>0$, given by
\begin{align}
	\Hil
	=
	L^2(\Rl,d\te),\qquad
	(U^{(m)}(x,\la)\psi)(\te)
	=
	e^{ip_m(\te)\cdot x}\psi(\te-\la),\quad p_m(\te)=
	m\begin{pmatrix}
		\cosh\te\\\sinh\te
	\end{pmatrix}.
\end{align}

The commutant of the diagonal representation $U^{(m)}_2$ on $L^2(\Rl^2)$ can be determined by general facts about massive representations, see Lemma 3.2 in \cite{LechnerLongo:2014}, or by a direct argument that we give below.

\begin{lemma}
	Let $m>0$. Then the representation $U^{(m)}_2$ satisfies
	\begin{align}
		\{{U^{(m)}_2}|_{\Rl^2}\}'
		=
		L^\infty(\Rl^2)\vee F.
	\end{align}
\end{lemma}
\begin{proof}
	The von Neumann algebra $\M=(U_2^{(m)}|_{\Rl^2})''$ is generated by the  multiplication operators multiplying with the symmetric functions $f_x(\te_1,\te_2)=\exp(ix\cdot(p_m(\te_1)+p_m(\te_2)))$, $x\in\Rl^2$. As $p_m(\te_1)+p_m(\te_2)=p_m(\te_1')+p_m(\te_2')$ is equivalent to $(\te_1,\te_2)=(\te_1',\te_2')$ or $(\te_1,\te_2)=(\te_2',\te_1')$, an application of the Stone-Weierstra\ss\, Theorem shows that $\M$ consists of all symmetric functions, i.e.\ $\M=L^\infty(\Rl^2)\cap\{F\}'$. This implies the claim by taking commutants.
\end{proof}

As an operator $T\in\CB(\Hil\ot\Hil)$ commuting with $U_2^{(m)}$ has also to commute with the boosts, acting as $(U^{(m)}_2(0,\la)\Psi)(\te_1,\te_2)=\Psi(\te_1-\la,\te_2-\la)$, it follows that it is of the form
\begin{align}\label{TForm}
	T=M_f^-+M_g^-F,\qquad
	(M_f^-\Psi)(\te_1,\te_2)
	=
	f(\te_2-\te_1)\Psi(\te_1,\te_2),
\end{align}
with $f,g\in L^\infty(\Rl)$.

We now consider a standard subspace $H\subset L^2(\Rl)$ that is natural in the context of quantum field, namely the one that was already discussed in the context of standard pairs in Section~\ref{sec:endos} and is given by
\begin{align}\label{WedgeH}
	(J_H\Psi)(\te)=\overline{\Psi(\te)},
	\qquad
	\Delta_H^{it}=U^{(m)}(0,2\pi t).
\end{align}
Recall that in this case, $\dom\Delta_H^{1/2}=H+iH=\Hardy^2(\Strip_\pi)$ consists of the (boundary values of) functions in the Hardy space $\Hardy^2(\Strip_\pi)$ on the strip $\Strip_\pi$. As before, we will write $\Hardy^\infty(\Strip_\pi)$ for the analytic bounded functions on this strip.

\begin{theorem}\label{thm:CrossUPoincare}
	Consider the irreducible positive energy representation $U^{(m)}$ of the two-dimensional Poincar\'e group with mass $m>0$, and the standard subspace~$H$ given by the modular data \eqref{WedgeH}.
	\begin{enumerate}
		\item An operator $T\in\{U^{(m)}_2\}'$ is $J_H$-crossable if and only if it is of the form
		\begin{align}
			T=
			M_g^-F,\,\qquad g\in L^\infty(\Rl).
		\end{align}
		It is $J_H$-crossing symmetric if and only if $g$ is real.
		
		\item An operator $T\in\{U^{(m)}_2\}'$ is $S_H$-crossable if and only if
		\begin{align}
			T=
			M_g^-F,\,\qquad g\in\Hardy^\infty(\Strip_{\pi}).
		\end{align}
		It is $S_H$-crossing symmetric if and only if $g(\te+i\pi)=\overline{g(\te)}$, and unitary if and only if $g\in \Hardy^\infty(\Strip_{\pi})$ is inner.
	\end{enumerate}
\end{theorem}

\begin{proof}
	a) Let $T=M_f^-+M_g^-F\in\{U_2\}'$. According to Example~\ref{example:basics}~c), $M_g^-F$ is $J_H$-crossable with $\Cross_{J_H}(M_g^-F)=M_{g_-}^-F$. As crossability is a linear condition, we only have to show that $M_f^-$ is not crossable.
	
	So let $\psi_1,\psi_2,\varphi_1,\varphi_2\in L^2(\Rl)$ and consider
	\begin{align*}
		Q_{J_H}(M_f^-)(\varphi_1\ot\varphi_2,\psi_1\ot\psi_2)
		&=
		\int_{\Rl^2} \overline{\varphi_1(\te_1)\varphi_2(\te_1)}f(\te_2-\te_1)\psi_1(\te_2)\psi_2(\te_2)\,d\te_1d\te_2.
	\end{align*}
	Extending from $\varphi_1\ot\varphi_2$ to $\Phi\in L^2(\Rl)\odot L^2(\Rl)$ and from $\psi_1\ot\psi_2$ to $\Psi\in L^2(\Rl)\odot L^2(\Rl)$ yields
	\begin{align*}
		Q_{J_H}(M_f^-)(\Phi,\Psi)
		&=
		\int_{\Rl^2} \overline{\Phi(\te_1,\te_1)}f(\te_2-\te_1)\Psi_1(\te_2,\te_2)\,d\te_1d\te_2
		,
	\end{align*}
	which is continuous in $\Phi$, $\Psi$ if and only if $f=0$. Hence $M_f^-$ is not crossable for $f\neq0$ (cf.~Lemma~\ref{lemma:crossability}~b)).
	
	b) We first use the necessary condition from Proposition~\ref{prop:KMS-likeCrossing}: If $T:=M_f^-+M_g^-F$ is $S_H$-crossable, then, for all $\psi_1,\ldots,\psi_4\in L^2(\Rl)$, the function $m:\Rl\to\Cl$,
	\begin{align*}
		m(t)&:=
		\int_{\Rl^2} \overline{\psi_1(\te_1)}\overline{\psi_2(\te_2+2\pi t)}\,f(\te_2-\te_1)\,\psi_3(\te_1+2\pi t)\psi_4(\te_2)\,d\te_1\,d\te_2
		\\
		&\qquad +
		\int_{\Rl^2} \overline{\psi_1(\te_1)}\overline{\psi_2(\te_2)}\,g(\te_2-\te_1+2\pi t)\,\psi_3(\te_2)\psi_4(\te_1)\,d\te_1\,d\te_2
	\end{align*}
	has an analytic continuation to the strip $\Strip_{1/2}$ and satisfies a uniform bound of the form $\sup_{\zeta\in\Strip_{1/2}}|m(\zeta)|
	\leq
	C\|\psi_1\|\|\psi_2\|\|\psi_3\|\|\psi_4\|$. 
	
	Denoting the two integral terms of $m$ by $m=I_f+I_g$, we first claim $I_f=0$. The reason for this is that for $\psi_2$ and $\psi_4$ compactly supported with disjoint supports, $I_f(t)$ vanishes in a neighbourhood of $t=0$. Hence, for those $\psi_2,\psi_4$, the second term $I_g$ must extend analytically to $\Strip_{1/2}$. Thus also $I_f$ must have an analytic continuation, i.e.\ we see $I_f(t)=0$ for all $t\in\Rl$. This implies $I_f(0)=0$ for all compactly supported $\psi_1,\ldots,\psi_4$, which forces $f=0$.
	
	The analyticity and boundedness of $I_g$ imply $g\in\Hardy^\infty(\Strip_\pi)$ (see Lemma A.4 in \cite{LongoWitten:2011} for similar techniques). According to Proposition \ref{prop:KMS-likeCrossing}, $I_g(\frac{i}{2})$ coincides with the expectation value $\langle J_H\psi_3\ot\psi_1,T^*(\psi_4\ot J_H\psi_2)\rangle$, which in view of $((M_g^-F)^*\Psi)(\te_1,\te_2)=\overline{g(\te_1-\te_2)}\Psi(\te_2,\te_1)$ is given by
	\begin{align*}
		\int_{\Rl^2}\psi_3(\te_1)\overline{\psi_1(\te_2)}\overline{g(\te_1-\te_2)}\psi_4(\te_2)\overline{\psi_2(\te_1)}\,d\te_1\,d\te_2.
	\end{align*}
	Comparison now yields $g(\te+i\pi)=\overline{g(\te)}$, as claimed.
	
	If conversely $T=M_g^-F$ with $g\in\Hardy^\infty(\Strip_\pi)$ satisfying this boundary condition, it is easy to check that $T$ is $S_H$-crossing symmetric.
\end{proof}

This theorem shows that the two body S-matrices used in quantum field theory (see, for example, \cite{Smirnov:1992,Lechner:2003}) exhaust all possibilities of Poincaré invariant crossing symmetric operators.

The examples discussed here exemplify the interplay of unitary representations and crossing symmetry. It is worth mentioning that \eqref{eq:klR} and the solutions described in Thm.~\ref{thm:CrossUPoincare} automatically satisfy the Yang--Baxter equation. We leave it as an interesting question to determine $\Hh^U_S$ for other natural pairs of Tomita operators and representations.

\section{The categorical crossing map}\label{sec:catcross}

In this section, we change our perspective and start from a real (or just rigid) object $X$ in an abstract $C^*$-tensor category $C$ to define a categorical crossing map $\CatCross_X$ (Definition \ref{def:CatCross}).
If $C$ is chosen to be the category of endomorphisms or bimodules of an infinite or tracial von Neumann factor, and $X$ is a Q-system (a particular real object used to describe subfactors), then the categorical crossing map is essentially the subfactor/planar algebraic Fourier transform, as introduced by Ocneanu \cite{Ocneanu:1991}. See Definition \ref{def:CatCross} and Remark \ref{rmk:CatCrossisF}. If $C$ is the category of finite-dimensional Hilbert spaces, then the categorical crossing map coincides with the concrete crossing map considered so far (here restricted to finite-dimensional $\Hil$) with respect to a (bounded) antilinear involution $S$. See Theorem \ref{thm:HilbfCrossing}. Later, we show how to produce a family of solutions $T$ of the crossing symmetry equation, $\CatCross_X(T) = T^*$, and (quantum) Yang--Baxter equation, $T_1T_2T_1 = T_2T_1T_2$, starting from an arbitrary Q-system (or $C^*$-Frobenius algebra) in $C$.
When $C$ is again finite-dimensional Hilbert spaces, we produce examples of $T$-twisted Araki--Woods von Neumann algebras in the sense of \cite{CorreaDaSilvaLechner:2023}, with cyclic and separating vacuum vector, generalizing an example of \cite{Yang:2023}.

For background on $C^*$-tensor categories, we refer to Sections 2 and 3 in \cite{LongoRoberts:1997}, Chapter 2 in \cite{NeshveyevTuset:2013}, Section 3.1 in \cite{BischoffKawahigashiLongoRehren:2015} and Section 2.10 in \cite{EtingofGelakiNikshychOstrik:2015}.
Unless otherwise mentioned, we use $C$ to denote a (strict) $C^*$-tensor category. Denote by $\id$ and $\otimes$ respectively the tensor unit object (assumed to be simple) and the tensor multiplication of $C$, and by $1_X$ the identity morphism from an object $X$ to itself. Morphisms in $C$ are denoted by $t:X\to Y$, or by $t\in\Hom_C(X,Y)$, where $\Hom_C(X,Y)$ is the space of morphisms from $X$ to $Y$ in $C$. The adjoint and norm of morphisms are respectively denoted by $t^*$ and~$\| t \|$.

Recall the definitions of rigid (or dualizable), self-conjugate (or self-dual), and real object.

\begin{definition}\label{def:rigidobject}
	An object $X$ in a $C^*$-tensor category $C$ is called \emph{rigid} or \emph{dualizable} if there is another object $\bar X$ in $C$, called \emph{conjugate} or \emph{dual} of $X$, and a pair of morphisms $\ev_X : \bar X \otimes X \to \id, \coev_X : \id \to X \otimes \bar X$ in $C$, respectively called \emph{evaluation} and \emph{coevaluation}, solving the so-called \emph{conjugate equations} for $(X, \bar X)$, namely
	\begin{align}\label{eq:conjeqns}
		(1_X \otimes \ev_X) (\coev_X \otimes 1_X) = 1_X, \quad(\ev_X \otimes 1_{\bar X}) (1_{\bar X} \otimes \coev_X) = 1_{\bar X}.
	\end{align}
	Note that if $\bar X$ is a conjugate of $X$ with $\ev_X, \coev_X$, then $X$ is a conjugate of $\bar X$ with $\ev_{\bar X} := \coev_X^*, \coev_{\bar X} := \ev_X^*$.
\end{definition}

\begin{remark}\label{rmk:dimension}
	When a conjugate object $\bar X$ exists, it is uniquely determined by $X$ up to unitary isomorphism. Moreover, a rigid $X$ has \emph{finite} \emph{intrinsic dimension} $d(X)$ in $C$. We refer to \cite[Sect.~3]{LongoRoberts:1997}, \cite[Chap. 2.2]{NeshveyevTuset:2013}, \cite[Chap. 2.2]{BischoffKawahigashiLongoRehren:2015} for more information on $d(X)$ and only note that for $C=\Hilbf$, the intrinsic dimension $d(X)$ coincides with the vector space dimension $\dim X$.
\end{remark}

For our definition of categorical crossing map, we recall the notion of self-conjugate and real objects (\cite[Sect.~5]{LongoRoberts:1997}, \cite[Remark 5.6 3)]{Mueger:2003-I}).

\begin{definition}\label{def:realobject}
	A rigid object $X$ is called \emph{self-conjugate} or \emph{self-dual} if $\bar X \cong X$, i.e.\ if $\bar X = X$ can be chosen. A self-conjugate object $X$ is called \emph{real} if $\ev_X$ and $\coev_X$ can be chosen such that $\coev_X = \ev_X^*$, i.e.\ $\coev_X = \coev_{\bar X}$ or equivalently $\ev_X = \ev_{\bar X}$.
\end{definition}

Note that for real $X$, either of the conjugate equations in \eqref{eq:conjeqns} is the adjoint of the other one.

Now we define the categorical crossing map $\CatCross_X$ associated with a rigid object $X$ in $C$. We shall focus on the case where $X$ is real.

\begin{definition}\label{def:CatCross}
	Let $C$ be a $C^*$-tensor category and let $X$ be a real object in $C$, together with $\bar X = X, \ev_X, \coev_X = \ev_X^*$ as in Definition \ref{def:realobject}.
	We define the \emph{categorical crossing map}
	\begin{align}
		\CatCross_X &: \Hom_C(X\otimes X, X\otimes X)\to \Hom_C(X\otimes X, X\otimes X),\\
		\CatCross_X(T) &:= (1_{X\otimes X} \otimes \ev_X) (1_X \otimes T \otimes 1_X) (\coev_X \otimes 1_{X\otimes X}).
	\end{align}
\end{definition}

If $X$ is only a rigid object with conjugate $\bar X$, one can analogously define a crossing map $\CatCross_X(T):\Hom_C(\bar X\otimes X, \bar X\otimes X)\to\Hom_C(X\otimes \bar X, X\otimes \bar X)$.

\begin{remark}\label{rmk:CatCrossisF}
	The reader familiar with subfactors and planar algebras will immediately recognise (by drawing a picture, cf.\ Figure \ref{fig:crossing}) the similarity between the map $\CatCross_X$ in Definition \ref{def:CatCross}, with $X$ either rigid or real in $C$, and the \emph{subfactor theoretical/planar algebraic Fourier transform}, see, e.g.\ Section II.7 in \cite{Ocneanu:1991} and Section 3 in \cite{BischJones:2000}. To make the relation precise, we anticipate that a particular class of real objects is given by Q-systems and C*-Frobenius algebras in $C$, see Definition \ref{def:Qsys} and \ref{def:CstarFrobAlg}. A Q-system $X$ (e.g.\ in $C = \End(\CN)$, the $C^*$-tensor category of endomorphisms of an infinite factor $\CN$ with intertwining operators) encodes the data of (a finite-index \cite{Jones:1983_2,Kosaki:1986} conditional expectation onto) a subfactor $\CN\subset\CM$, see Section 5 in \cite{Longo:1990,Longo:1994}, cf.\ Section 7 in \cite{Giorgetti:2022}.
	Let $X = \bar Y \otimes Y$ be a splitting of $X:\CN\to\CN$ (the object of the Q-system) by means of $Y:\CN\to\CM$ (the inclusion morphism) and $\bar Y:\CM\to\CN$ (its conjugate morphism), see e.g.\ Theorem 3.11 in \cite{BischoffKawahigashiLongoRehren:2015} and Theorem 3.36 in \cite{ChenHernandezPalomaresJonesPenneys2022}.
	Then the subfactor Fourier transform maps $\Hom_{\End(\CN)}(\bar Y \otimes Y, \bar Y \otimes Y)$ onto $\Hom_{\End(\CM)}(Y \otimes \bar Y, Y \otimes \bar Y)$ by exactly the same formula used for $\CatCross_Y$, with $Y$ rigid, but not real, in the 2-$C^*$-category of morphisms between $\CN$ and $\CM$ or vice versa. 
	Thus, $\CatCross_Y$ is the ordinary (two-shaded, finite-index) subfactor Fourier transform (or its inverse, depending on the conventions). Subfactorially, it maps the two-step relative commutant of $\CN$ in the Jones tower generated by $\CN \subset \CM$ onto the one of $\CM$ (i.e.\ it maps one 2-box space of the subfactor planar algebra onto the other). Instead, $\CatCross_X$ maps morphisms in $C=\End(\CN)$ to itself (it is unshaded). It maps the four-step relative commutant of $\CN$ onto itself.
	For further references on the subfactor Fourier transform in various contexts (finite or infinite subfactors with either finite or infinite index, paragroups, planar algebras, fusion bialgebras, etc.), see e.g.\ \cite{NillWiesbrock:1995, Bisch:1997,Sato:1997,JiangLiuWu:2016,JaffeJiangLiuRenWu:2020,LiuPalcouxWu:2021,BischoffDelVecchioGiorgetti:2021,BischoffDelVecchioGiorgetti:2022}.
\end{remark}

The following definition makes sense only for self-conjugate objects.

\begin{definition}\label{def:CatCrossingSymmet}
	Let $X$ be a self-conjugate object in $C$ with a chosen $\ev_X$ and $\coev_X(=\ev_X^* \textrm{ if $X$ is real})$.
	We call an element $T \in \Hom_C(X\otimes X, X\otimes X)$ \emph{crossing symmetric} if it fulfills
	\begin{align}\label{eq:catcrossingsymmet}
		\CatCross_X(T) = T^*.
	\end{align}
\end{definition}

To explain the connection of $\CatCross_X$ with the Hilbert space crossing map $\Cross_S$, we specialize to $C = \Hilbf$, the $C^*$-tensor category of finite-dimensional complex Hilbert spaces with complex linear maps, and tensor unit $\id = \bC$. As is well known, every object in $\Hilbf$ is real, see Definition \ref{def:realobject}. Denote the Hilbert space scalar product by $\langle \cdot , \cdot \rangle_\Hil$ or just $\langle \cdot , \cdot \rangle$. The theorem below establishes in particular a characterization of standard subspaces of finite-dimensional Hilbert spaces in terms of categorical data in $\Hilbf$.

\begin{theorem}\label{thm:HilbfCrossing}
	Let $C = \Hilbf$ and let $\Hil$ be an object in $\Hilbf$.
	\begin{enumerate}
		\item There is a one-to-one correspondence between solutions of the conjugate equations \eqref{eq:conjeqns} with $\coev_\Hil = \ev_\Hil^*$ and antilinear involutions $S\in\CS(\Hil)$. This correspondence is given by
		\begin{align}\label{eq:defSfromrealHil}
			\ev_\Hil (v \otimes w) &= \langle S v, w\rangle, \quad v,w \in \Hil,\\
			\label{eq:coevS}
			\coev_\Hil(\la)&=\la\cdot\xi_S=
			\la\sum_n e_n\ot Se_n,\qquad \la\in\Cl,
		\end{align}
		where $\{e_n\}_n$ is an arbitrary orthonormal basis of $\Hil$, and $\xi_S$ is as in \eqref{def:xiS}.
		\item Let $S\in\CS(\Hil)$ and consider the solutions to the conjugate equations defined by the bijection in a). Then the corresponding categorical crossing map $\CatCross_\Hil$ coincides with the Hilbert space crossing map $\Cross_S$.
	\end{enumerate}
\end{theorem}
\begin{proof}
	a) In $\Hilbf$, the evaluation map $\ev_\Hil : \Hil \otimes \Hil \to \bC$ is a bilinear form on~$\Hil$. Hence there exists a unique antilinear map $S: \Hil \to \Hil$ (not a morphism in $\Hilbf$) such that \eqref{eq:defSfromrealHil} holds.
	
	The adjoint $\coev_\Hil = \ev_\Hil^*$ is easily computed: For $\la\in\Cl$, $v,w\in\Hil$, and some orthonormal basis $\{e_n\}_n$ of $\Hil$, we have
	\begin{align}
		\langle v\ot w,\coev_\Hil(\la)\rangle
		&=
		\la\overline{\langle Sv,w\rangle}
		=
		\la\sum_n\langle v,e_n\rangle\langle e_n,S^*w\rangle
		=
		\langle v\ot w,\la\sum_n e_n\ot Se_n\rangle.
	\end{align}
	This proves \eqref{eq:coevS} and shows that $\xi_S$ depends only on $S$ and not on the choice of orthonormal basis. The (equivalent) conjugate equations \eqref{eq:conjeqns} now read, $v\in\Hil$,
	\begin{align}
		v=
		(1_\Hil\ot\ev_\Hil)(\coev_\Hil\ot1_\Hil)v
		&=
		(1_\Hil\ot\ev_\Hil)(\xi_S\ot v)
		=
		\sum_ne_n\langle S^2e_n,v\rangle
		=
		(S^*)^2v,
	\end{align}
	and clearly hold if and only if $S$ is an involution.
	
	We thus see that any solution $\coev_\Hil=\ev_\Hil^*$ of the conjugate equations \eqref{eq:conjeqns} defines a unique involution $S\in\CS(\Hil)$. Conversely, given arbitrary $S\in\CS(\Hil)$, we may define $\ev_\Hil$ via \eqref{eq:defSfromrealHil} and set $\coev_\Hil:=\ev_\Hil^*$. As these are morphisms satisfying the conjugate equations, we have established the claimed one-to-one correspondence.
	
	b) Let $A,B\in\CB(\Hil)$ and $T:=(A\ot B)F\in\CB(\Hil\ot\Hil)$. Then we have, $v,w\in\Hil$,
	\begin{align*}
		\CatCross_\Hil(T)(v\ot w)
		&=
		(1\ot1\ot\ev_\Hil)(1\ot T\ot1)(\coev_\Hil\ot1\ot1)(v\ot w)
		\\
		&=
		(1\ot1\ot\ev_\Hil)(1\ot T\ot1)(\xi_S\ot v\ot w)
		\\
		&=
		\sum_n
		(1\ot1\ot\ev_\Hil)(1\ot A\ot B\ot1)(e_n\ot v\ot Se_n\ot w)
		\\
		&=
		\sum_n
		(e_n\ot Av)\langle SBS e_n,w\rangle
		\\
		&=
		(S^*B^*S^*\ot A)F(v\ot w).
	\end{align*}
	Thus $\CatCross_\Hil(T)=(S^*B^*S^*\ot A)F$, which coincides with $\Cross_S(T)$ by Example~\ref{example:basics}~b). Since any $T\in\CB(\Hil\ot\Hil)$ is a linear combination of maps of the form $(A\ot B)F$, the proof is finished.
\end{proof}

\begin{remark}
	\leavevmode
	\begin{enumerate}
		\item The evaluation and coevaluation maps in \eqref{eq:defSfromrealHil} and \eqref{eq:coevS} capture the crossing of the identity as described in Example~\ref{example:basics}~e), namely $\Cross_S(1)=|\xi_S\rangle\langle\xi_S|$.
		
		\item Theorem~\ref{thm:HilbfCrossing} does not directly generalize to infinite-dimensional Hilbert spaces because given an involution $S\in\CS(\Hil)$, the corresponding evaluation map $\Hil\ot\Hil\ni v\ot w\mapsto\langle Sv,w\rangle$ is bounded if and only if $\dim\Hil<\infty$.
		
		\item Theorem~\ref{thm:HilbfCrossing} establishes a bijection between antilinear involutions $S\in\CS(\Hil)$ and solutions $\coev_\Hil^*=\ev_\Hil$ of the conjugate equations. As $\CS(\Hil)$ is in bijection with the set of all real standard subspaces in $\Hil$, or the set of all real structures on $\Hil$, it also establishes the characterization of standard subspaces of finite-dimensional Hilbert spaces in terms of $\ev$, $\coev = \ev^*$ solving the conjugate equations \eqref{eq:conjeqns} in $\Hilbf$.
		
		\item From the definition of $S$ in \eqref{eq:defSfromrealHil} it follows immediately that $S = S^*$ (hence $S$ is antiunitary, $\Delta = 1_\Hil$ and $S = J$) if and only if $\ev_\Hil : \Hil \otimes \Hil \to \bC$ is symmetric, i.e.\ if and only if $\ev_\Hil\circ F = \ev_\Hil$, where $F : \Hil\otimes \Hil \to \Hil\otimes \Hil$ is the tensor flip (the unitary and symmetric braiding of $C=\Hilbf$).
		
		If the given solution $\ev_\Hil$, $\coev_\Hil = \ev_\Hil^*$ of the conjugate equations \eqref{eq:conjeqns} is \emph{standard} in the sense of Section 3 in \cite{LongoRoberts:1997}, cf.\ Definition 2.2.14 in \cite{NeshveyevTuset:2013}, then necessarily $\ev_\Hil \circ F = \ev_\Hil$ by the much more general results in Section 4 in \cite{LongoRoberts:1997}, cf.\ Section 7 in \cite{DoplicherRoberts:1989-1}.
		Hence, we conclude that if $\ev_\Hil$, $\coev_\Hil = \ev_\Hil^*$ is standard, then $\Delta = 1_\Hil$ and $\CatCross_\Hil = \Cross_J$, as a special case of Theorem \ref{thm:HilbfCrossing}.
		
		Assuming standardness, the claimed equalities $\Delta = 1_\Hil$ and $\CatCross_\Hil = \Cross_J$ can also be checked directly. We only need to know that the standard solutions of the conjugate equations \eqref{eq:conjeqns} for $(\Hil,\bar \Hil)$ in $C=\Hilbf$, where $\bar \Hil$ is the complex conjugate Hilbert space, are given by (up to unitaries in $\Hom_C(\Hil,\Hil)$, cf.\ Example 2.2.2 in \cite{NeshveyevTuset:2013})
		\begin{align}\label{eq:evcoevstandardinHilbf}
			\ev_\Hil : \bar v \otimes w &\mapsto \langle v, w\rangle_\Hil \\
			\coev_\Hil : \lambda \in \bC &\mapsto \lambda \sum_i e_i \otimes \bar e_i
		\end{align}
		where $\{e_i\}_i$ runs over an arbitrary orthonormal basis of $\Hil$. Moreover, by combining \eqref{eq:defSfromrealHil} and \eqref{eq:evcoevstandardinHilbf} we conclude that the converse is also true, namely $\Delta = 1_\Hil$ if and only if $\ev_\Hil$, $\coev_\Hil = \ev_\Hil^*$ is standard.
	\end{enumerate}
\end{remark}

Back to the general setting of Definition \ref{def:CatCross}, we look at $C^*$-tensor categories different from $\Hilbf$ and their crossing maps.
Starting from a Q-system or $C^*$-Frobenius algebra (defined below) in an arbitrary $C^*$-tensor category $C$,  we find a family of solutions of the crossing symmetry equation $\CatCross_X(T) = T^*$ as in Definition \ref{def:CatCrossingSymmet}, that furthermore solve the (quantum) Yang--Baxter equation $T_1 T_2 T_1 = T_2 T_1 T_2$. If $C=\Hilbf$, the previous solutions are concrete linear operators on $\Hil\otimes\Hil$, and we show that they are invariant under the modular group $\Delta^{it} \otimes \Delta^{it}$, $t\in\bR$, as in Definition \ref{def:G-invariant}.

\begin{definition}\label{def:Qsys}
	Let $C$ be a $C^*$-tensor category. A \emph{Q-system} in $C$ is a triple $(X,m,\iota)$ where $X$ is an object in $C$, $m : X \otimes X \to X$ and $\iota : \id \to X$ are morphisms in $C$, respectively called multiplication and unit of $X$, such that
	\begin{align*}
		m (m \otimes 1_X) &= m (1_X \otimes m),  &\text{(associativity)}\\
		m (\iota \otimes 1_X) &= 1_X = m (1_X \otimes \iota), &\text{(unitality)}\\
		(m \otimes 1_X) (1_X \otimes m^*) &= m^* m = (1_X \otimes m)(m^* \otimes 1_X), &\text{(}C^*\text{-Frobenius property})\\
		mm^* &= d 1_X, \text{   for some } d>0  &\text{(specialness)}\\
		\iota^* \iota &= 1_{\id}. &\text{(normalization)}
	\end{align*}
\end{definition}

\begin{remark}
	There are several redundancies in the definition above.
	Also, it is known from \cite{LongoRoberts:1997}, see also Lemma 3.7 in \cite{BischoffKawahigashiLongoRehren:2015}, that specialness implies the $C^*$-Frobenius property. Moreover, the normalization condition can always be achieved by rescaling $\iota$ and $m$, as $\iota^*\iota : \id \to \id$, hence $\iota^*\iota = \lambda 1_{\id}$, for some $\lambda > 0$, by the simplicity of the tensor unit of $C$. Hence the essential ingredients for a Q-system are associativity, unitality (the algebra structure, see below), and specialness (the unitary separability property, cf.\ \cite{GiorgettiYuanZhao:2024}).
\end{remark}

We refer the reader e.g.\ to \cite{BischoffKawahigashiLongoRehren:2015,HeunenVicary:2019,ChenHernandezPalomaresJonesPenneys2022,CarpiGaudioGiorgettiHillier:2023}, and references therein, for the motivation of the study of Q-systems, initially tied to finite-index subfactors \cite{Longo:1994}, and for the subsequent manifold applications. 

\begin{definition}\label{def:CstarFrobAlg}
	A triple $(X,m,\iota)$ in $C$ as above fulfilling only associativity and unitality is called an \emph{algebra} in $C$. If $(X,m,\iota)$ fulfills in addition the C*-Frobenius property, it is called a \emph{$C^*$-Frobenius algebra} in $C$.
\end{definition}

\begin{remark}
	By Lemma 3.5 in \cite{BischoffKawahigashiLongoRehren:2015}, see also Lemma 2.10 in \cite{CarpiGaudioGiorgettiHillier:2023}, for every algebra in $C$, $mm^* : X \to X$ is invertible (hence strictly positive). If the algebra is normalized, i.e.\ $\iota^* \iota = 1_{\id}$, then $mm^* \geq 1_X$ in $\Hom_C(X,X)$.
	Moreover, every $C^*$-Frobenius algebra in $C$ is isomorphic to a Q-system by suitably rescaling $\iota$ and $m$ with $(\iota^*\iota)^{1/2}$ and $(mm^*)^{1/2}$, cf.\ Corollary 3.6 in \cite{BischoffKawahigashiLongoRehren:2015}.
\end{remark}

The most natural examples of $C^*$-Frobenius algebra objects are the following:

\begin{example}
	Every finite-dimensional Hopf $C^*$-algebra \cite{Longo:1994,SzymanskiPeligrad:1994}, cf.\ \cite{Woronowicz:1987, VaesVanDaele:2001}, or Hopf $^*$-algebra coaction on a finite-dimensional $C^*$-algebra defines a special $C^*$-Frobenius algebra object in $C = \Hilbf$ by Lemma 2.2 in \cite{AranoDeCommer:2019}, cf.\ \cite{LarsonSweedler:1969,Pareigis:1971}.
\end{example}

\begin{example}\label{ex:CstarFrobAlginHilbf}
	More generally, a $C^*$-Frobenius algebra object (not necessarily special, nor normalized) in $C = \Hilbf$ is the same as a concrete finite-dimensional Frobenius $C^*$-algebra by Lemma 2.2 in \cite{NeshveyevYamashita:2018}, cf.\ \cite{Abrams:1999}. We further comment on this at the end of the section.
\end{example}

The object $X$ of any $C^*$-Frobenius algebra is rigid (Definition \ref{def:rigidobject}) and real (Definition \ref{def:realobject}) in $C$. Indeed, the morphisms
\begin{align}
	\ev_X &:= \iota^* m : X \otimes X \to \id \\
	\coev_X &:= m^* \iota : \id \to X \otimes X
\end{align}
solve the conjugate equations \eqref{eq:conjeqns}, just by using the $C^*$-Frobenius property and unitality, and they fulfill $\coev_X = \ev_X^*$ with $\bar X = X$.

Q-systems fulfill also $\ev_X \coev_X = \iota^* m m^* \iota = d 1_{\id}$, where $d$ is the positive constant entering in the definition of specialness. In this case, by the results of \cite{LongoRoberts:1997}, we necessarily have that $d \geq d(X)$, where $d(X)$ is the intrinsic dimension of $X$ as a rigid object in $C$, cf.\ Remark \ref{rmk:dimension}. Indeed, $d(X)$ is the minimal possible value arising from solutions of the conjugate equations and it is related to the minimal index for subfactors \cite{Kosaki:1986, Hiai:1988, Longo:1989}.

\begin{definition}
	A Q-system $(X,m,\iota)$ in $C$ is called \emph{standard} if $d = d(X)$, or equivalently, if $\ev_X \coev_X = d(X) 1_{\id}$.
	A Q-system is called \emph{irreducible} (or haploid, or connected) if the vector space of morphisms $\Hom_C(\id, X)$ is one-dimensional, i.e.\ $\Hom_C(\id,X)=\bC\iota$.
\end{definition}

\begin{remark}
	As observed in Section 6 in \cite{LongoRoberts:1997}, see also Remark 5.6, 3) in \cite{Mueger:2003-I} and Theorem 2.9 in \cite{NeshveyevYamashita:2018}, an irreducible Q-system in $C$ is automatically standard. Furthermore, by Lemma 3.3 in \cite{BischoffKawahigashiLongoRehren:2015}, see also Lemma 2.13 in \cite{CarpiGaudioGiorgettiHillier:2023}, an irreducible $C^*$-Frobenius algebra in $C$ is automatically a (standard) Q-system.
\end{remark}

Below, starting from Q-systems and $C^*$-Frobenius algebras in an abstract $C^*$-tensor category $C$, we provide positive solutions $T\in \Hom_C(X\otimes X, X\otimes X)$ of the crossing symmetry equation $\CatCross_X(T) = T^*$ (Definition \ref{def:CatCrossingSymmet}) and of the (quantum) Yang--Baxter equation $T_1 T_2 T_1 = T_2 T_1 T_2$. 

\begin{example}
	Let $(X,m,\iota)$ be a $C^*$-Frobenius algebra in $C$. A trivial family of solutions of crossing symmetry and Yang--Baxter is given by $T := (\iota \otimes \iota) (\iota^* \otimes \iota^*)$, as one can easily check. Note that $T = T^*\geq 0$ and ($= T^2$ if $\iota$ is normalized).
\end{example}

A more interesting family of solutions comes as follows. Let $(X,m,\iota)$ be a $C^*$-Frobenius algebra in $C$. Let $T := m^*m$. Then $T = T^* \geq 0$. If $(X,m,\iota)$ is also a Q-system, then $T^2 = d T$ and $\|T\| = d$, where $d>0$ is the constant entering in the definition of specialness. In particular, $d^{-1} T$ is a self-adjoint projection, cf.\ Remark \ref{rmk:Jonesproj}.
Examples of such $T = m^*m$ coming from Q-systems have been implicitly considered in Example 2.12 in \cite{Yang:2023}, in the case $C=\Hilbf$ and $X=M_n(\Cl)$, the algebra of $n\times n$ complex matrices, cf.\ Corollary \ref{cor:twistedAWalgs} and the examples thereafter.

\begin{proposition}\label{prop:TsolvesCross}
	Let $(X,m,\iota)$ be a $C^*$-Frobenius algebra in $C$.
	Then $T:=m^*m$ is crossing symmetric in the sense of Definition \ref{def:CatCrossingSymmet}, namely it fulfills
	\begin{align}
		\CatCross_X(T) = T^* = T.
	\end{align}
\end{proposition}

\begin{proof}
	Compute
	\begin{align}
		(1_{X\otimes X} \otimes \ev_X)& (1_X \otimes T \otimes 1_X) (\coev_X \otimes 1_{X\otimes X}) \\
		&= (1_{X\otimes X} \otimes (\iota^* m)) (1_X \otimes (m^*m) \otimes 1_X) ((m^* \iota) \otimes 1_{X\otimes X}) \\
		&= (1_{X\otimes X} \otimes (\iota^* m)) (1_X \otimes m^* \otimes 1_X) (1_X \otimes m \otimes 1_X) ((m^* \iota) \otimes 1_{X\otimes X}) \\
		&= (1_{X} \otimes 1_{X} \otimes \iota^*) (1_{X} \otimes 1_{X} \otimes m) (1_X \otimes m^* \otimes 1_X) \\
		& \hspace{3.7cm}(1_X \otimes m \otimes 1_X) (m^* \otimes 1_{X} \otimes 1_{X}) (\iota \otimes 1_{X} \otimes 1_{X}) \\
		&= (1_{X} \otimes 1_{X} \otimes \iota^*) (1_{X} \otimes (m^*m)) ((m^*m) \otimes 1_X) (\iota \otimes 1_{X} \otimes 1_{X}) \\
		&= (1_{X} \otimes ((1_{X} \otimes \iota^*)m^*m)) ((m^*m (\iota \otimes 1_{X})) \otimes 1_X) \\
		&= (1_{X} \otimes m) (m^* \otimes 1_X) \\
		&= m^*m = T.
	\end{align}
\end{proof}

\begin{proposition}\label{prop:TsolvesYB}
	Let $(X,m,\iota)$ be a Q-system in $C$.
	Then $T := m^*m$ fulfills the categorical (quantum) Yang--Baxter equation, $T_1 T_2 T_1 = T_2 T_1 T_2$, namely
	\begin{align}
		(T \otimes 1_X) (1_X \otimes T) (T \otimes 1_X) = (1_X \otimes T) (T \otimes 1_X) (1_X \otimes T).
	\end{align}
\end{proposition}

We first show the following lemma.

\begin{lemma}\label{lem:TsolvesGL}
	Let $T$ be as above (in fact, we only need the $C^*$-Frobenius property, associativity and co-associativity of $(X,m,\iota)$). Then
	\begin{align}
		(T \otimes 1_X) (1_X \otimes T) = (1_X \otimes T) (T \otimes 1_X).
	\end{align}
\end{lemma}

\begin{proof}
	\begin{align}
		(T \otimes 1_X) (1_X \otimes T) &= ((m^*m) \otimes 1_X) (1_X \otimes (m^*m))\\
		&= (m^* \otimes 1_X) (m \otimes 1_X) (1_X \otimes m^*) (1_X \otimes m) \\
		& = (m^* \otimes 1_X) m^* m (1_X \otimes m) \\
		& = (1_X \otimes m^*) m^* m (m \otimes 1_X) \\
		& = (1_X \otimes m^*) (1_X \otimes m) (m^* \otimes 1_X) (m \otimes 1_X) \\
		& = (1_X \otimes (m^*m)) ((m^*m) \otimes 1_X) = (1_X \otimes T) (T \otimes 1_X).
	\end{align}
\end{proof}

From this, the Yang--Baxter equation easily follows.

\begin{proof}(of Proposition \ref{prop:TsolvesYB}).
	By Lemma \ref{lem:TsolvesGL}, we have
	\begin{align}
		(T \otimes 1_X) (1_X \otimes T) (T \otimes 1_X) = (1_X \otimes T) (T^2 \otimes 1_X) = d (1_X \otimes T) (T \otimes 1_X),
	\end{align}
	and
	\begin{align}
		(1_X \otimes T) (T \otimes 1_X) (1_X \otimes T) = (T \otimes 1_X) (1_X \otimes T^2) = d (T \otimes 1_X) (1_X \otimes T).
	\end{align}
	The two expressions are the same again by Lemma \ref{lem:TsolvesGL}.
\end{proof}

\begin{remark}\label{rmk:Jonesproj}
	One may wonder if, given a Q-system $(X,m,\iota)$ in a $C^*$-tensor category $C$, the \emph{Jones projection} $E_X := d^{-1}\coev_X \ev_X = d^{-1} m^* \iota \iota^* m$ also provides a solution to the equations in the previous Propositions \ref{prop:TsolvesCross} and \ref{prop:TsolvesYB}.
	The projections $E_X = E_X^* = E_X^2$ and $1_{X \otimes X}$ can be combined in order to give solutions of the Yang--Baxter equation. Indeed, $1_X \otimes E_X$ and $E_X \otimes 1_X$ fulfill the stronger Temperley--Lieb algebra relations \cite{Jones:1985}. However, $E_X$ fulfills the crossing symmetry equation if and only if $E_X = 1_{X\otimes X}$, namely, if and only if $(X=\id, m=\iota=1_{\id}, d=1)$ is the trivial Q-system in $C$.
	Indeed, $\CatCross_X(1_{X\otimes X}) = \coev_X \ev_X = d E_X$, an equation we have already encountered in \eqref{eq:CrossId} for $C = \Hilbf$, $X=\Hil$ of dimension $N$, and $d=N$.
\end{remark}

\begin{remark}
	In the case of concrete Frobenius algebras, where the (not necessarily $C^*$) tensor category $C$ is vector spaces or modules over a commutative ring, the solutions $T=m^*m$ of Yang--Baxter provided in Proposition \ref{prop:TsolvesYB} are not new. They already appeared in their concrete form in \cite{BeidarFongStolin:1997,CaenepeelBogdanMilitaru:2000}, inspired by \cite{Drinfeld:1983}. In the context of subfactors, where the tensor category $C$ is either endomorphisms or bimodules of an infinite or tracial factor, cf.\ Remark \ref{rmk:CatCrossisF}, the Yang--Baxter solutions of Proposition \ref{prop:TsolvesYB} do not seem to be straightforwardly related to the Temperley--Lieb algebra relations of the subfactor associated with the input Q-system, cf.\ Remark \ref{rmk:Jonesproj} and see \cite{Jones:1985, GoodmandelaHarpeJones:1989}. For example Lemma \ref{lem:TsolvesGL} does not hold for the Jones projection $E_X$, as the left and right wave diagrams $(E_X \otimes 1_X) (1_X \otimes E_X)$ and $(1_X \otimes E_X) (E_X \otimes 1_X)$ are not equal.
\end{remark}

\begin{proposition}\label{prop:TisDeltaitinvariant}
	Let $(X,m,\iota)$ be a $C^*$-Frobenius algebra (in particular a real object with $\ev_X = \iota^* m = \coev_X^*$) in $C = \Hilbf$. Let $S$ be the antilinear involution defined by \eqref{eq:defSfromrealHil} with polar decomposition $S = J\Delta^{1/2}$. Then $T := m^*m$ is invariant under the modular group $\Delta^{it}$, $t\in\bR$, as in Definition \ref{def:G-invariant}. Namely
	\begin{align}
		[T,\Delta^{it}\ot \Delta^{it}]=0, \quad t\in \bR.
	\end{align}
\end{proposition}

We first show a lemma, whose first statement is already contained in Lemma 2.2 in \cite{NeshveyevYamashita:2018}, cf.\ the proof of Proposition 3.4 in \cite{AranoDeCommer:2019}. See also Lemma 4.1 in \cite{Vicary:2011}.

\begin{lemma}\label{lem:Sisalgebrainvol}
	Let $(X,m,\iota)$ be a $C^*$-Frobenius algebra and $C=\Hilbf$. Consider the representation by left multiplication of the algebra $X$ on itself (as a Hilbert space), namely $L_x y := m(x\otimes y)$, $x,y \in X$. Consider also the representation by right multiplication of the opposite algebra $X$ (with opposite multiplication $m^\mathrm{op} := m F$, where $F$ is the tensor flip), namely, $R_x y := m(y\otimes x)$, $x,y \in X$.
	
	Then $(L_x)^* = L_{Sx}$ and $(R_x)^* = R_{S^*x}$, or equivalently
	\begin{align}
		\langle m( x \otimes y), z \rangle = \langle y, m(Sx \otimes z) \rangle, \quad x,y,z \in X
	\end{align}
	and
	\begin{align}
		\langle m( y \otimes x), z \rangle = \langle y, m(z \otimes S^*x) \rangle, \quad x,y,z \in X.
	\end{align}
	
	In particular, $S$ and $S^*$ are algebra involutions, namely
	\begin{align}
		S m = m F (S\otimes S) \quad \mathrm{and} \quad S^* m = m F (S^*\otimes S^*).
	\end{align}
\end{lemma}

\begin{proof}
	First, observe that the $C^*$-Frobenius relation implies (and is, in fact, equivalent to) $m^* = (1_X \otimes m) (\coev_X \otimes 1_X)$ where $\coev_X = m^* \iota$ (in an arbitrary $C^*$-tensor category $C$). In $C = \Hilbf$, recall also the computation after \eqref{eq:defSfromrealHil} for $\coev_X = \sum_i e_i \otimes S e_i$, where $\{e_i\}_i$ is an arbitrary orthonormal basis of $X$. Then
	\begin{align}
		\langle y, m(Sx \otimes z) \rangle &= \langle m^*(y), Sx \otimes z \rangle \\
		&= \langle (1_X \otimes m) (\coev_X \otimes\, y), Sx \otimes z \rangle \\
		&= \sum_i \langle (1_X \otimes m) (e_i \otimes S e_i \otimes\, y), Sx \otimes z \rangle \\
		&= \sum_i \langle e_i, Sx \rangle \langle m(S e_i \otimes\, y), z \rangle \\
		&= \sum_i \langle m(S (\langle e_i, Sx \rangle e_i) \otimes\, y), z \rangle \\
		&= \langle m(S^2 x \otimes\, y), z \rangle \\
		&= \langle m(x \otimes\, y), z \rangle
	\end{align}
	because $\langle \cdot , \cdot \rangle$ is antilinear on the left and $S$ is an antilinear involution. The first algebra involution property $S(m(x\otimes y)) = m(Sy \otimes Sx)$, $x,y \in X$, follows. The second adjunction relation $\langle m( y \otimes x), z \rangle = \langle y, m(z \otimes S^*x) \rangle$ together with $S^*(m(x\otimes y)) = m(S^*y \otimes S^*x)$, $x,y \in X$, can be proven similarly.
\end{proof}

\begin{proof}(of Proposition \ref{prop:TisDeltaitinvariant}).
	By Lemma \ref{lem:Sisalgebrainvol}, and by direct computation, $T (\Delta \otimes \Delta) = (\Delta \otimes \Delta) T$ where $T = m^*m$ and $\Delta = S^*S$ follows.
	The desired equality $T (\Delta^{it}\ot \Delta^{it}) = (\Delta^{it}\ot \Delta^{it}) T$, $t\in \bR$, follows too by functional calculus.
\end{proof}

With the last three Propositions \ref{prop:TsolvesCross}, \ref{prop:TsolvesYB}, \ref{prop:TisDeltaitinvariant} together with the concrete form $\Cross_S$ (Definition \ref{def:Scrossable}) of the categorical crossing map $\CatCross_X$ (Definition \ref{def:CatCross}) for $C=\Hilbf$ given by Theorem \ref{thm:HilbfCrossing}, and the previous Proposition \ref{prop:KMS-likeCrossing} on crossing symmetry and KMS property, we get the following.

\begin{corollary}\label{cor:twistedAWalgs}
	Every concrete Q-system $(X,m,\iota)$ in $C=\Hilbf$ (see Example \ref{ex:CstarFrobAlginHilbf} and Proposition \ref{prop:QsysCstar} below) produces a positive strict twist $T := m^*m : X\otimes X \to X\otimes X$ which is compatible, crossing symmetric and braided in the language of \cite{CorreaDaSilvaLechner:2023}, acting on the finite-dimensional space $X\otimes X$. In particular, each such $T$ gives an examples of a $T$-twisted Araki--Woods von Neumann algebra realized on $T$-twisted Fock space with (cyclic and) separating vacuum vector, see e.g.\ Corollary 3.23 in \cite{CorreaDaSilvaLechner:2023}, where the real standard subspace $H \subset X$ is determined by the real structure $\ev = \iota^* m = \coev^*$ and by the corresponding antilinear involution $S\in\CS(X)$.
\end{corollary}

We conclude by providing more concrete examples of Q-systems $(X,m,\iota)$ in $C=\Hilbf$, either standard, i.e.\ with $\Delta = 1_X$, or with non-trivial $\Delta$, i.e.\ where $S \neq S^*$.
Note that if $m$ is commutative or if $\iota^*$ is tracial, then certainly $\ev_X (= \iota^* m)$ is symmetric and $\Delta$ is trivial.

For each example below we compute the operators $\Delta$ and $J$ (on $X$), and $T := m^*m$ (acting linearly on $X\otimes X$) solving the crossing symmetry equation (Proposition \ref{prop:TsolvesCross}), Yang--Baxter equation (Proposition \ref{prop:TsolvesYB}), and modular group invariance equation $T (\Delta^{it} \otimes \Delta^{it}) = (\Delta^{it} \otimes \Delta^{it}) T$, $t\in\bR$ (Proposition \ref{prop:TisDeltaitinvariant}).

\begin{example}\label{ex:QsysC(G)}
	The algebra of complex-valued functions over a finite group $G$ with (commutative) pointwise multiplication gives a Q-system $(C(G) \cong \bC G, m, \iota)$ in $C = \Rep_f^u(G)$, the category of finite-dimensional unitary representations of $G$. In fact, an irreducible standard Q-system.
	The *-structure on $\Rep_f^u(G)$ is given by the embedding into $\Hilbf$. Explicitly, as $G$-invariant scalar product on $C(G)$ we take the normalized $L^2$ scalar product $\langle \psi_1, \psi_2 \rangle := \int_G \bar \psi_1 \psi_2 dg = |G|^{-1} \sum_g \bar \psi_1(g) \psi_2(g)$.
	The algebra structure on $C(G)$ is given by
	\begin{alignat}{2}
		m &: \, C(G) \otimes C(G) \to C(G), \qquad m^* &&: \, C(G) \to C(G) \otimes C(G), \\
		& \, \psi_1 \otimes \psi_2 \mapsto \psi_1 \psi_2 &&\, \psi \mapsto |G| \sum_g \psi(g) \delta_g \otimes \delta_g \\
		\iota &: \, \bC \to C(G), \hspace{31.5mm} \iota^* &&: \, C(G) \to \bC \\
		& \, \lambda \mapsto \lambda 1 &&\, \psi \mapsto |G|^{-1} \sum_g \psi(g).
	\end{alignat}
	These linear maps are easily seen to be intertwiners of the $G$-actions by (left) translation $(_g\psi)(h) := \psi(g^{-1}h)$, hence they are morphisms in $\Rep_f^u(G)$. Also, $\iota^* \iota = 1_{\bC}$ (normalization) and $mm^* = |G| 1_{C(G)}$ (specialness) hold. Either by observing that the Q-system is irreducible (the only $G$-invariant functions are the constants), or by $d(C(G)) = |G|$, we have that $(C(G), m, \iota)$ is a standard Q-system.
	
	We also see that $\ev_{C(G)} := \iota^* m : \psi_1 \otimes \psi_2 \mapsto \langle \bar \psi_1, \psi_2 \rangle$ and $\coev_{C(G)} := m^* \iota : \lambda \mapsto \lambda \sum_g |G|^{1/2} \delta_g \otimes |G|^{1/2} \delta_g$. Namely, by identifying the conjugate Hilbert space $\overline{C(G)}$ with $C(G)$ using the antiunitary involution $J: \psi \mapsto \bar\psi$, we get as $\ev_{C(G)}$ and $\coev_{C(G)}$ the same standard solutions as in \eqref{eq:evcoevstandardinHilbf}.
	
	The associated operator $T := m^*m : C(G) \otimes C(G) \to C(G) \otimes C(G)$ acts by $\delta_g \otimes \delta_h \mapsto |G| \delta_g(h)\, \delta_g \otimes \delta_g$ on the basis $\delta_g \otimes \delta_h$.
\end{example}

\begin{example}\label{ex:QsysL1(G)}
	The convolution algebra of functions over a finite group $G$ also gives an irreducible standard Q-system $(L^1(G) = \bC G, m, \iota)$ in $C = \Hilbf_{,G}$, the category of finite-dimensional $G$-graded Hilbert spaces. Similarly to the previous example, take the non-normalized $L^2$ scalar product $\langle \psi_1, \psi_2 \rangle := \sum_g \bar \psi_1(g) \psi_2(g)$, which preserves the $G$-grading ($\delta_g$ and $\delta_h$ are orthogonal for $g\neq h$ and the scalar product is the direct sum scalar product of $|G|$ copies of $\bC$). The algebra structure is given by
	\begin{alignat}{2}
		m &: \, L^1(G) \otimes L^1(G) \to L^1(G), \qquad\quad m^* &&: \, L^1(G) \to L^1(G) \otimes L^1(G), \\
		m(\psi_1 \otimes &\psi_2)(h) := \sum_g \psi_1(g) \psi_2(g^{-1}h) &&\, m^*(\psi)(g,h) := \psi(gh) \\
		\iota &: \, \bC \to L^1(G), \hspace{38.5mm} \iota^* &&: \, L^1(G) \to \bC \\
		& \, \lambda \mapsto \lambda \delta_e &&\, \psi \mapsto \psi(e).
	\end{alignat}
	The maps are linear and preserve the $G$-grading (the convolution of $\delta_g$, $\delta_h$ is $\delta_{gh}$), hence they are morphisms in $\Hilbf_{,G}$. Normalization and irreducibility are immediate. Specialness and standardness follow by $m^*(\psi) = \sum_{g,h} \psi(h) \delta_g \otimes \delta_{g^{-1}h}$ and $mm^* = |G| 1_{L^1(G)}$ as in the previous example.
	
	By identifying $L^1(G)$ with $\overline{L^1(G)}$ using the antiunitary involution $J: \psi \mapsto J\psi$, $(J\psi)(g) := \overline{\psi(g^{-1})}$, we get again that $\ev_{L^1(G)} := \iota^* m : \psi_1 \otimes \psi_2 \mapsto \langle J\psi_1, \psi_2 \rangle$ and $\coev_{L^1(G)} := m^* \iota : \lambda \mapsto \lambda \sum_g \delta_g \otimes J\delta_g$ are the standard solutions in \eqref{eq:evcoevstandardinHilbf}.
	
	The associated operator $T := m^*m : L^1(G) \otimes L^1(G) \to L^1(G) \otimes L^1(G)$ acts by $\delta_g \otimes \delta_h \mapsto \sum_k \delta_k \otimes \delta_{k^{-1}gh}$ on the basis $\delta_g \otimes \delta_h$.
\end{example}

We now come back to $C^*$-Frobenius algebras and Q-systems $(X,m,\iota)$ in $C = \Hilbf$, as already mentioned in Example \ref{ex:CstarFrobAlginHilbf}. We first look at their general structure (see Theorem 4.6 in \cite{Vicary:2011}, Lemma 2.2 in \cite{NeshveyevYamashita:2018}, and cf.\ Proposition 3.4 in \cite{AranoDeCommer:2019}).

\begin{proposition}\label{prop:QsysCstar}
	There is a one-to-one correspondence between normalized $C^\ast$-Frobenius algebras $(X,m,\iota)$ in $\Hilbf$ and pairs $(\CA,\om)$ of finite-dimensional $C^\ast$-algebras $\CA$ and faithful states $\om$ on $\CA$. This correspondence identifies $X$ with $\CA$ as unital associative algebras, $\iota^*$ with $\om$ as functionals, and the involution $S$ defined by $\ev_X:=\iota^*m$ as in \eqref{eq:defSfromrealHil} with the Tomita operator of $(\CA,\om)$ in its GNS representation.
\end{proposition}
\begin{proof}
	A $C^*$-Frobenius algebra $(X,m,\iota)$ in $\Hilbf$ is in particular a unital associative algebra, and $\iota^*:X\to\Cl$ is a linear functional. We will equip $X$ with the structure of a $C^*$-algebra so that $\om:=\iota^*$ will become a faithful state.
	
	We define $x^\dagger:=S(x)$, $x\in X$, where $S$ is the antilinear involution defined by \eqref{eq:defSfromrealHil} with respect to $\ev_X=\iota^*m$. As the left multipliers satisfy $(L_x)^*=L_{x^\dagger}$ by Lemma~\ref{lem:Sisalgebrainvol}, this turns $X$ into a unital ${}^\dagger$-algebra. The norm $\|x\|:=\|L_x\|_{\CB(X)}$ is a $C^\dagger$-norm on it, and we denote the resulting $C^\dagger$-algebra $\CA$. Now $\om:=\iota^*$ satisfies $\om(x^\dagger y)=\iota^*(m(S(x)\ot y))=\langle x,y\rangle$ and is hence a faithful state on $\CA$.
	
	Conversely, let $\CA$ be a finite-dimensional $C^\dagger$-algebra with faithful state $\om$. In view of the faithfulness and finite dimension, the GNS construction $(\pi,\Hil,\Om)$ turns $\CA$ into a Hilbert space $X=\Hil$, with scalar product $\langle x,y\rangle:=\om(x^\dagger y)$. Clearly $X$ inherits an associative multiplication $m:X\ot X \to X$ with unit $\iota:\Cl\to X$, $\la\mapsto\la1_\CA$ (or just $1_\CA = 1$) from $\CA$. We have $\iota^*=\om$, and hence the normalization of $\om$ implies $\iota^*\iota=\id_{\bC}$.
	
	It remains to establish the $C^*$-Frobenius property for $(X,m,\iota)$. To that end, let $\{v_n\}_n$ denote an orthonormal basis of the GNS space $\Hil$. Since the GNS vector $\Om$ is cyclic and separating for $\pi$, there exist unique elements $e_n\in\CA$ such that $\pi(e_n)\Om=v_n$, and we consider
	\begin{align}
		f:=\sum_n e_n\ot e_n^\dagger \in \CA\ot\CA.
	\end{align}
	In the following, we will identify $\CA$ and $\CA\ot\CA$ with their (faithful) GNS representations w.r.t.\ $\om$ and $\om\ot\om$, respectively, and write $S$ for the Tomita operator of $(\CA,\om)$.
	
	We claim $m^*(x)=(x\ot 1)f$, $x\in\CA$, where $f$ satisfies $f(\Om\ot\Om)=\xi_S$ in the notation of \eqref{def:xiS}. For any $a,b,x\in\CA$, we have
	\begin{align*}
		\langle a\Om\ot b\Om,\,(x\ot1)f(\Om\ot\Om)\rangle
		&=
		\sum_n \langle x^\dagger a\Om,v_n\rangle\langle b\Om,Sv_n\rangle
		\\
		&=
		\langle x^\dagger a\Om,S^*b\Om\rangle
		\\
		&=
		\langle b\Om,a^\dagger x\Om\rangle
		\\
		&=\langle m(a\ot b)\Om,x\Om\rangle
		=
		\langle a\Om\ot b\Om,m^*(x)(\Om\ot\Om)\rangle.
	\end{align*}
	As $\Om\ot\Om$ is cyclic and separating for $\CA\ot\CA$, this proves the claimed formula $m^*(x)=(x\ot1)f$ and shows in particular that $f$ is independent of the chosen orthonormal basis.
	
	The $C^*$-Frobenius property now follows by a routine calculation: For any $a,b\in\CA$, we have
	\begin{align*}
		(m\ot\id)(\id\ot m^*)(a\ot b)
		&=
		(m\ot\id)(a \otimes (b\ot 1) f)
		\\
		&=
		\sum_n abe_n\ot e_n^\dagger
		=
		(ab\ot 1)f
		=
		(m^*m)(a\ot b).
	\end{align*}
	As the two constructions are clearly inverses of each other, the proof is finished.
\end{proof}

With notation as in this proof, we set $z:=\sum_n e_ne_n^\dagger\in\CA$, and observe
\begin{align}\label{eq:z}
	(mm^*)(x)
	&=
	m((x\ot1)f)
	=
	\sum_n xe_ne_n^\dagger
	=
	x\cdot z,\qquad x\in\CA.
\end{align}
Thus the $C^*$-Frobenius algebras described in Proposition \ref{prop:QsysCstar} are special (i.e. Q-systems) if and only if $z$ is a multiple of the identity. To describe this more explicitly, and to arrive at concrete representing matrices, we now consider a general finite-dimensional $C^*$-algebra $\CA$ with a general faithful state $\om=\Tr(\rho\,\cdot\,)$, namely
\begin{align}\label{eq:concrete-A-omega}
	\CA=\bigoplus_\alpha X_\alpha,\qquad \rho=\bigoplus_\alpha\rho_\alpha,
\end{align}
where $X_\alpha=M_{n_\alpha}(\bC)$, the sums run for $\alpha=1, \ldots N$, and the $\rho_\alpha\in X_\alpha$ are positive and invertible, with $\sum_\alpha\Tr\rho_\alpha=1$. We shall also use the flip operator in $X_{\alpha}\ot X_{\alpha}$, namely $F_{\alpha}:=\sum_{i,j=1}^{n_\alpha}E^{(\alpha)}_{ij}\ot E^{(\alpha)}_{ji}$, where $E_{ij}^{(\alpha)}$ are the matrix units for $X_\alpha$.

\begin{corollary}
	Consider the normalized $C^*$-Frobenius algebra $(X,m,\iota)$ corresponding to the data \eqref{eq:concrete-A-omega} via Proposition \ref{prop:QsysCstar}.
	\begin{enumerate}
		\item  $(X,m,\iota)$ is special (a Q-system) if and only if there exists $d>0$ such that $\Tr(\rho_\alpha^{-1})=d$ for all $\alpha$.
		\item $(X,m,\iota)$ is standard if and only if $\rho_\alpha = n_\alpha/{\sum_{\alpha=1}^N n_\alpha^2}\, 1_{X_\alpha}$.
		\item The associated crossing symmetric operator $T := m^*m : X \otimes X \to X \otimes X$ acts on pure tensors by $x_\alpha \otimes y_\beta \mapsto \delta_{\alpha,\beta}\, (x_\alpha y_\beta \otimes \rho_\beta^{-1}) F_\alpha$ for every $x_\alpha\in X_\alpha$, $y_\beta\in X_\beta$.
	\end{enumerate}
\end{corollary}
\begin{proof}
	
	The GNS representation of $(\CA,\om)$ has the form
	\begin{align*}
		\Hil_\om=\bigoplus_\alpha\Cl^{n_\alpha}\ot\Cl^{n_\alpha},\quad
		\Om=\bigoplus_\alpha\sum_{k=1}^{n_\alpha}\sqrt{\la^{(\alpha)}_k}v_k^{(\alpha)}\ot v_k^{(\alpha)},\quad
		\pi(\bigoplus_\alpha x_\alpha) = \bigoplus_\alpha\left(x_\alpha\ot1_{X_\alpha}\right),
	\end{align*}
	where $\{v_k^{(\alpha)}\}_k$ is an orthonormal basis of $\Cl^{n_\alpha}$ in which $\rho_\alpha$ is diagonal, with eigenvalues $\la^{(\alpha)}_k$.
	
	In the notation from the proof of Proposition \ref{prop:QsysCstar}, we therefore consider the algebra elements $e_{\alpha,i,j}=(\la^{(\alpha)}_j)^{-1/2}E^{(\alpha)}_{ij}$, and have
	\begin{align}
		f&=\bigoplus_\alpha \sum_{i,j=1}^{n_\alpha}\frac{1}{\la_j^{(\alpha)}} E^{(\alpha)}_{ij}\ot E^{(\alpha)}_{ji}
		=
		\bigoplus_\alpha (1_{X_\alpha}\ot\rho_\alpha^{-1})F_{\alpha},
		\\
		z&=\bigoplus_\alpha \sum_{i,j=1}^{n_\alpha}\frac{1}{\la_j^{(\alpha)}} E_{ij} E_{ji}
		=
		\bigoplus_\alpha \Tr(\rho_\alpha^{-1})1_{X_\alpha}.
	\end{align}
	Comparing $z$ with $d\cdot 1_\CA$ and $d(\CA)=\sum_\alpha n_\alpha^2$ now yields a) and b), and c) follows from the formula for $f$.
\end{proof}

We thus see that Q-systems $(X,m,\iota)$ in $\Hilbf$ generate explicit solutions to the Yang--Baxter equation that are crossing symmetric w.r.t.\ the Tomita operator given by \eqref{eq:defSfromrealHil}, which here is encoded in the density matrix $\rho$. We leave it to a future work to find further connections between Q-systems and crossing symmetric solutions to the Yang--Baxter equation, possibly in the setting of subfactors.

\subsubsection*{Acknowledgements}
L.G. thanks Sebastiano Carpi for insightful feedback on concrete $C^*$-Frobenius algebras, and G.L. thanks Dietmar Bisch for discussions on subfactor theoretical Fourier transforms.

\end{document}